\documentclass[12pta4paper]{amsart}
\usepackage{amsfonts,amssymb,amsthm,amsmath,hyperref, color}
\newtheorem{theorem}{Theorem}[section]
\newtheorem{corollary}[theorem]{Corollary}
\newtheorem{question}[theorem]{Question}

\newtheorem{lemma}[theorem]{Lemma}
\newtheorem{proposition}[theorem]{Proposition}

\newtheorem{definition}[theorem]{Definition}

\newcommand{\IR}{\mathbb R}

\newcommand{\U}{\mathcal U}
\newcommand{\w}{\omega}

\newcommand{\Op}{\mathcal{O}\/}
\newcommand{\B}{\mathcal{B}}
\newcommand{\C}{\mathcal{C}}

\newcommand{\K}{\mathcal{K}}
\newcommand{\A}{\mathcal{A}}
\newcommand{\D}{\mathcal{D}}

\newcommand{\CS}{\mathcal{S}}

\newcommand{\J}{\mathcal{J}}
\newcommand{\M}{\mathcal{M}}
\newcommand{\NN}{\mathcal{N}}

\newcommand{\F}{\mathcal{F}}

\newcommand{\uhr}{\upharpoonright}

\newcommand{\la}{\langle}
\newcommand{\ra}{\rangle}

\newcommand{\hot}{\mathfrak}
\newcommand{\W}{\mathcal W}

\newcommand{\nothing}[1]{}

\newcommand{\Rnw}{R-$\mathit{nw}$-selective}

\newcommand{\wgp}{\w\mbox{-}\mathit{gp}}

\title{On some recent selective properties involving networks}
\author{Maddalena Bonanzinga$^1$, Davide Giacopello$^2$, Santi Spadaro$^3$, and Lyubomyr Zdomskyy$^4$}
\date{}

\begin{document}

 \maketitle

\begin{abstract}
	In this paper we investigate R-,H-, and M-{\it nw}-selective properties introduced in \cite{BG}. In particular, we provide consistent uncountable examples of such spaces and we define \textit{trivial} R-,H-, and M-{\it nw}-selective spaces the ones with countable net weight having, additionally, the cardinality and the weight strictly less then $cov(\M)$, $\frak b$, and $\frak d$, respectively. Since we establish that spaces having cardinalities more than $cov(\M)$, $\frak b$, and $\frak d$, fail to have the R-,H-, and M-{\it nw}-selective properties, respectively, non-trivial examples should eventually have weight greater than or equal to these small cardinals.	Using forcing methods, we construct consistent countable non-trivial examples of R-{\it nw}-selective and H-{\it nw}-selective spaces and we establish some limitations to constructions of non-trivial examples. Moreover, we consistently prove the existence of two H-{\it nw}-selective spaces whose product fails to be M-{\it nw}-selective. Finally, we study some relations between {\it nw}-selective properties and a strong version of the HFD property.
\end{abstract}

{\bf Keywords:} Countable net weight; M-separable space, H-separable space, R-separable space,
Menger space,
Hurewicz space,
Rothberger space.

{\bf AMS Subject Classification:}  54D65, 54A25, 54A20.

%%%%%%%%%%%%%%%%%%%%%%%%%%%%%%%%
\section{Introduction}
%%%%%%%%%%%%%%%%%%%%%%%%%%%%%%%
\setcounter{footnote}{0}
\renewcommand{\thefootnote}{}
\footnotetext{$^1$MIFT Department, University of Messina, Italy; Email address: mbonanzinga@unime.it}
\footnotetext{$^2$MIFT Department, University of Messina, Italy; Email address: dagiacopello@unime.it}
\footnotetext{$^3$University of Palermo, Italy; Email address: santidomenico.spadaro@unipa.it}
\footnotetext{$^4$Institut f\"ur Diskrete Mathematik und Geometrie, Technische Universit\"at Wien, Wiedner Hauptstrasse 8-10/104, 1040 Wien, Austria; Email address: lzdomsky@gmail.com; URL: \href{https://dmg.tuwien.ac.at/zdomskyy/}{personal website}.}

\renewcommand{\thefootnote}{\arabic{footnote}}

	%\affil[2]{University of Messina, Italy
	%Email address: dagiacopello@unime.it}
	%\affil[3]{University of Palermo, Italy}
	%\affil[4]{Institut f\"ur Diskrete Mathematik und Geometrie, Technische Universit\"at Wien, 
	%Wiedner Hauptstrasse 8-10/104, 1040 Wien, Austria.\\
	%Email address: lzdomsky@gmail.com\\
	%URL: https://dmg.tuwien.ac.at/zdomskyy/}}

The systematic study of covering properties began with Scheepers \cite{Sch1}. Later, Scheepers himself and many other mathematicians (for instance see \cite{ BBM, BBMT, BMS, BCPT, Sch1, Sch, Sch2, Sch3, SchTsa02}) used these methods to describe other properties involving some other topological objects, not just collections of coverings of some type.

This type of new approach has led to catalog these properties within so-called \lq\lq selection principles\rq\rq. In particular, given two collections $\A$ and $\B$ of some topological objects on a  space $X$, Scheepers introduce this notation:
\begin{itemize}
	\item[$S_1(\A,\B)$]: For every sequence $(\U_n: n\in\omega)$ of elements of $\A$ there exists $U_n\in \U_n$, $n\in\omega$, such that $\{U_n:n\in\omega\}$ belongs to $\B$.
	\item[$S_{fin}(\A,\B)$]: For every sequence $(\U_n: n\in\omega)$ of elements of $\A$ there exists a finite subset $\F_n\in \U_n$, $n\in\omega$, such that $\bigcup_{n\in\omega} \F_n$ belongs to $\B$.
	\item[$U_{fin}(\A,\B)$]: For every sequence $(\U_n: n\in\omega)$ of elements of $\A$ there exists a finite subset $\F_n\subseteq \U_n$, $n\in\omega$, such that $\{\bigcup\F_n : n\in\omega\}$ belongs to $\B$.
\end{itemize}

Recall that a $\gamma$-cover of a space $X$ is a particular cover such that each point of $X$ belongs to all but finitely many members of the cover. If one denotes by $\Op$ and $\Gamma$ the family of all open covers and the family of all $\gamma$-covers of a space $X$, respectively, it follows that the \emph{Rothberger} property can be expressed by $S_1(\Op,\Op)$, the  \emph{Menger} property by $S_{fin}(\Op,\Op)$ and the  \emph{Hurewicz} property by $U_{fin}(\Op,\Gamma)$.
Menger \cite{H}, Rothberger \cite{R}, and Hurewicz \cite{H1} properties are classical topological properties that were among the first to be introduced as fundamental selection principles. These properties play a crucial role in understanding the combinatorial aspects of covering properties in topological spaces.

Inspired by the previous selective variation of Lindel\"ofness, many mathematicians introduced and studied some selection principles that are strengthening of separability (see for instance \cite{BBMT, BBM, Sch2}).

We denote by $\D$ the family of all dense subsets of a space $X$ and by $\D_\Gamma$ the family of all collections of subsets $\F$ such that for every nonempty open set $O\subset X$, the intersection $O\cap F$ is nonempty for all but finitely many $F\in \F$. A space $X$ is called  \emph{R-separable} if it satisfies $S_1(\D,\D)$,  \emph{M-separable} a space sastisfying $S_{fin}(\D,\D)$ and  \emph{H-separable} a space satisfying $U_{fin}(\D,\D_\Gamma)$. The \lq\lq M-\rq\rq, \lq\lq R-\rq\rq, and \lq\lq H-\rq\rq are motived by analogy with Menger, Rothberger, and Hurewicz properties, respectively.
\smallskip

Recall that for $f,g\in \omega^\omega$, $f\leq^*g$ means that $f(n)\leq g(n)$ for all but finitely many $n$ (and $f\leq g$ means that $f(n)\leq g(n)$ for all $n\in\omega$). A subset $B\subseteq \omega^\omega$ is \emph{bounded} if there is $g\in\omega^\omega$ such that $f\leq^*g$ for every $f\in B$. $D\subseteq \omega^\omega$ is \emph{dominating} if for each $g\in \omega^\omega$ there is $f\in D$ such that $g\leq^*f$. The minimal cardinality of an unbounded subset of $\omega^\omega$ is denoted by ${\hot b}$, and the minimal cardinality of a dominating subset of $\omega^\omega$ is denoted by ${\hot d}$. The value of ${\hot d}$ does not change if one considers the relation $\leq$ instead of $\leq^*$ \cite[Theorem 3.6]{vD}. $\mathcal M$ denotes the family of all meager subsets of $\mathbb R$. $cov(\mathcal M)$ is the minimum of the cardinalities of subfamilies ${\mathcal U}\subseteq {\mathcal M}$ such that $\bigcup{\mathcal U}=\IR$. However, another description of the cardinal $cov(\mathcal M)$ is given by the following.
\begin{theorem}{\rm(\cite{B} and \cite[Theorem 2.4.1]{BJ})} \label{thm:cov-char} %\rm
	$cov(\mathcal M)$ is the minimum cardinality of a family $F\subset \omega^\omega$ such that for every $g\in \omega^\omega$ there is $f\in F$ such that $f(n)\not =g(n)$ for all but finitely many $n\in \omega$.
\end{theorem}
Thus if $F\subset \omega^\omega$ and $|F|<cov(\mathcal M)$, then there is $g\in \omega^\omega$ such that for every $f\in F$, $f(n) = g(n)$ for infinitely many $n\in\omega$; it is often said that $g$ \textit{guesses} $F$.
For a topological property $\mathcal K$, non($\mathcal K$) denotes the minimum cardinality of a subspace of $\mathbb R$ that does not have property $\mathcal K$. It is well known that non(Menger) $= \hot d$, non(Hurewicz) $= \hot b$, non(Rothberger) $= cov({\mathcal M})$ (see \cite{MF, JMS}). Additionally every Lindel\"{o}f space of cardinality strictly less than $\hot d$ ($\hot b$ or $cov(\mathcal M)$) is  Menger (Hurewicz or Rothberger respectively).

\smallskip
A family $\mathcal P$ of open sets is called a $\pi$-base for $X$ if every nonempty open set in $X$ contains a nonempty element of $\mathcal P$; $\pi w(X)=min\{|{\mathcal P}|: {\mathcal P}$ is a $\pi$-base for $X\}$ is the $\pi$-weight of $X$. Scheepers \cite{Sch2} showed that every space with countable $\pi$-weight is R-separable (hence M-separable). Actually, in \cite{Sch2} it was proved that having countable $\pi$-weight is equivalent to a stronger property defined in terms of topological games. Also in \cite{BBMT} it is observed that every space with countable $\pi$-weight is H-separable.
Let $\delta(X) = sup\{d(Y ): Y$ is dense in $X\}$ \cite{WS}; $\delta(X) = \omega$ for every M-separable space X.
If $\delta(X) = \omega$  and 
$\pi w(X) < {\frak d}$, then $X$ is M-separable (a stronger version of this fact is estabilished in \cite[Theorem 40]{Sch2}); moreover,
if $\delta(X) = \omega$  and 
$\pi w(X) <$ cov$(\M)$, then $X$ is R-separable (a stronger version of this fact is estabilished in \cite[Theorem 29]{Sch2}); in \cite[Theorem 29]{BBM} it is also shown that 
if $\delta(X) = \omega$  and 
$\pi w(X) < {\frak b}$, then $X$ is H-separable. As a consequence of these results, it is shown that the existence of a countable M-separable space which is not H-separable is consistent with ZFC \cite{BBM}.

The following implications are obvious.

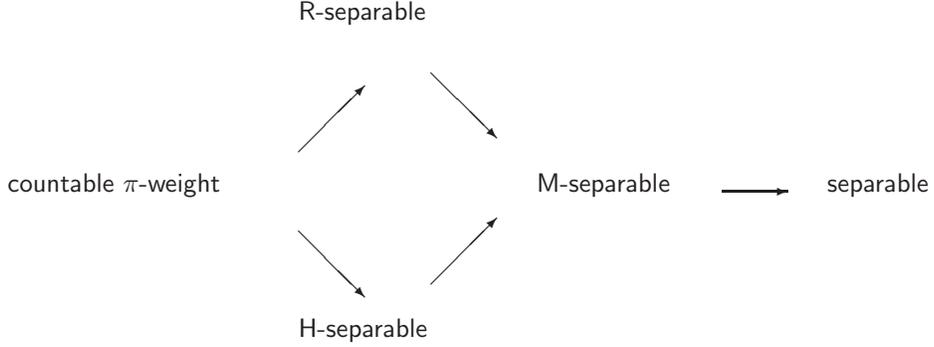
\begin{picture}(300,180)
\put(0,85){\textsf{{countable $\pi$-weight}}} 
\put(200,85){\textsf{{M-separable}}} 
\put(110,100){\vector(1,1){25}}
\put(110,150){\textsf{{R-separable}}}
\put(110,30){\textsf{{H-separable}}}
\put(110,70){\vector(1,-1){25}} 
\put(160,130){\vector(1,-1){25}}
\put(160,50){\vector(1,1){25}}
\put(270,85){\vector(1,0){25}}
\put(310,85){\textsf{{separable}}} 
\end{picture}

For compact spaces, M-, R- and H- separability are equivalent to each other and to having a countable $\pi$-base (see \cite{BBMT}).

%%%%%%%%%%%%%%%%%%%%%%%%%%%%%%%
%Recall that a space $X$ is Menger if for every sequence $({\mathcal U}_n : n\in\omega)$ of open covers of $X$ one can select finite ${\mathcal F}_n\subset {\mathcal U}_n$, $n\in\omega$, such that $\bigcup_{n\in\omega} {\mathcal F}_n$ covers $X$; $X$ is Rothberger if for every sequence $({\mathcal U}_n : n\in\omega)$ of open covers of $X$ one can select  $F_n\in {\mathcal U}_n$, $n\in\omega$, such that $\{F_n : n\in\omega\}$ covers $X$; $X$ is Hurewicz if for every sequence $({\mathcal U}_n : n\in\omega)$ of open covers of $X$ one can select finite ${\mathcal F}_n\subset {\mathcal U}_n$, $n\in\omega$, such that for every $x\in X$, $x\in \bigcup {\mathcal F}_n$ for all but finitely many $n$.\\
%The previous definitions motivated the introduction in \cite{Sch2} and \cite{BBM} of the following similar selection principles that are strengthening of separability. A space $X$ is M-separable (in \cite{BBMT} it is called selective separable) if for every sequence $(D_n : n\in\omega)$ of dense subsets of $X$ one can pick finite $F_n\subset D_n$, $n\in\omega$, such that $\bigcup_{n\in\omega} F_n$ is dense in $X$;  $X$ is R-separable if for every sequence $(D_n : n\in\omega)$ of dense subsets of $X$ one can pick $p_n\in D_n$, $n\in\omega$, such that $(p_n: n\in\omega)$ is dense in $X$; $X$ is H-separable if for every sequence $(D_n : n\in\omega)$ of dense subsets of $X$ one can pick finite $F_n\subset D_n$, $n\in\omega$, such that for every nonempty open set $O\subset X$, the intersection $O\cap F_n$ is nonempty for all but finitely many $n$.\\
Recall the following properties which can also be presented using the notation introduced by Scheepers.

A space $X$ has  \emph{countable (strong) fan tightness}, (see \cite{Arh2} and \cite{S0}), if for every point $x\in X$ and every sequence $(A_n:n\in\omega)$ of subspaces of $X$ such that $x\in\overline{A_n}$ for all $n\in\omega$, one can choose finite $F_n\subset A_n$ (resp., a point $x_n\in A_n$) so that $x\in\overline{\bigcup\{F_n: n\in\omega\}}$ (resp., $x\in\overline{\{x_n: n\in\omega\}}$).
It is natural to say that $X$ has  \emph{countable (strong) fan tightness with respect to dense subspaces} if this previous statement is true when the $A_n$'s are dense in $X$. %, that is for every $x\in X$ and every sequence $(A_n:n\in\omega)$ of dense subspaces of $X$ one can choose finite $F_n\subset A_n$ so that $x\in\overline{\bigcup\{F_n: n\in\omega\}}$.
A space $X$ is  \emph{weakly Fr\'echet in the strict sense} if for every point $x\in X$ and every sequence $(A_n:n\in\omega)$ of subspaces of $X$ such that $x\in\overline{A_n}$ for all $n\in\omega$, there are finite $F_n\subset A_n$ such that every neighborhood of $x$ intersects all but finitely many $F_n$ \cite{S1}. Again, the space $X$ is  \emph{weakly Fr\'echet in the strict sence with respect to dense subspaces} if this previous statement is true when the $A_n$'s are dense in $X$ \cite{BBM}. %, that is for every sequence $(D_n:n\in\omega)$ of dense subspaces of $X$ and every $x\in X$ there are finite $F_n\subset D_n$ such that every neighborhood of $x$ intersects all but finitely many $F_n$ .

Any separable space having countable (strong) fan tightness with respect to dense subsets is M-separable (R-separable, respectively); moreover, any separable weak Frechet in the strict sense with respect to dense subsets space is H-separable (see \cite{BBM} and \cite{BBMT}).

%Let $\delta(X) = sup\{d(Y ): Y$ is dense in $X\}$ \cite{WS}; $\delta(X) = \omega$ for every M-separable space X.
%If $\delta(X) = \omega$  and $\pi w(X) < {\hot d}$, then $X$ is M-separable (a stronger version of this fact is estabilished in \cite[Theorem 40]{Sch2}); moreover, if $\delta(X) = \omega$  and $\pi w(X) <$ cov$({\mathcal M})$, then $X$ is R-separable (a stronger version of this fact is estabilished in \cite[Theorem 29]{Sch2}); in \cite[Theorem 29]{BBM} it is also shown that if $\delta(X) = \omega$  and $\pi w(X) < {\hot b}$, then $X$ is H-separable.\\ 
%As a consequence of these results, in \cite{BBM}, the authors showed, consistently with ZFC, the existence of a countable M-separable space which is not H-separable.\\  
%Every space having a countable $\pi$-base is R-separable and H-separable, then M-separable. However,  not every space with countable net weight is M-separable. 
A family $\mathcal N$ of sets is called a network for $X$ if for every $x\in X$ and for every open neighbourhood $U$ of $x$ there exists an element $N$ of $\mathcal N$ such that $x\in N\subseteq U$; $nw(X)=min\{|{\mathcal N}|: {\mathcal N}$ is a network for $X\}$ is the net weight of $X$.

Bonanzinga and Giacopello \cite{BG} asked which additional conditions a space with countable netweight must satisfy in order to be M-separable (either R-separable or H-separable). They introduced and studied the following classes of spaces which constitute a combination of the covering type selection principles and the selection properties which are variations of separability. These properties thus represent, in some way, one of the strongest generalizations among the existing selection principles \cite{BG}.

A space $X$ is \emph{M-{\it nw}-selective} if $nw(X)=\omega$ and for every sequence $({\mathcal N}_n : n\in\omega)$ of countable networks for $X$ one can select finite ${\mathcal F}_n\subset {\mathcal N}_n$, $n\in\omega$, such that $\bigcup_{n\in\omega} {\mathcal F}_n$ is a network for $X$; $X$ is  \emph{R-{\it nw}-selective} if $nw(X)=\omega$ and for every sequence $({\mathcal N}_n : n\in\omega)$ of countable networks for $X$ one can pick  $N_n\in {\mathcal N}_n$, $n\in\omega$, such that $\{N_n :n\in\omega\}$ is a network for $X$; and $X$ is  \emph{H-{\it nw}-selective} if $nw(X)=\omega$ and for every sequence $({\mathcal N}_n : n\in\omega)$ of countable networks for $X$ one can select finite ${\mathcal F}_n\subset {\mathcal N}_n$, $n\in\omega$, such that for any $x\in X$ and any open neighbourhood $U$ of $x$, there exists some $\kappa\in\omega$ such that for any $n\geq\kappa$ there exists $A\in {\mathcal F}_n$ with $x\in A\subseteq U$.

In \cite{BG} it was proved that any R-{\it nw}-selective (M-{\it nw}-selective and H-{\it nw}-selective, respectively) space has countable strong fan tightness (countable fan tightness and the weak Frechet in strict sense property, respectively) hence it is R-separable (M-separable and H-separable, respectively).

The following diagram sums up all the implications that exist between these properties.

%In \cite[Proposition 2.3]{BBMT} it is shown that every space with countable $\pi$-weight has all this properties as the following diagram shows.
%
%\begin{picture}(300,180)
%\put(0,85){\textsf{{countable $\pi$-weight}}} 
%\put(200,85){\textsf{{M-separable}}} 
%\put(110,100){\vector(1,1){25}}
%\put(110,150){\emph{} \textsf{{R-separable}}}
%\put(110,30){\emph{} \textsf{{H-separable}}}
%\put(110,70){\vector(1,-1){25}} 
%\put(160,130){\vector(1,-1){25}}
%\put(160,50){\vector(1,1){25}}
%\put(270,85){\vector(1,0){25}}
%\put(310,85){\textsf{{separable}}} 
%\end{picture}

	
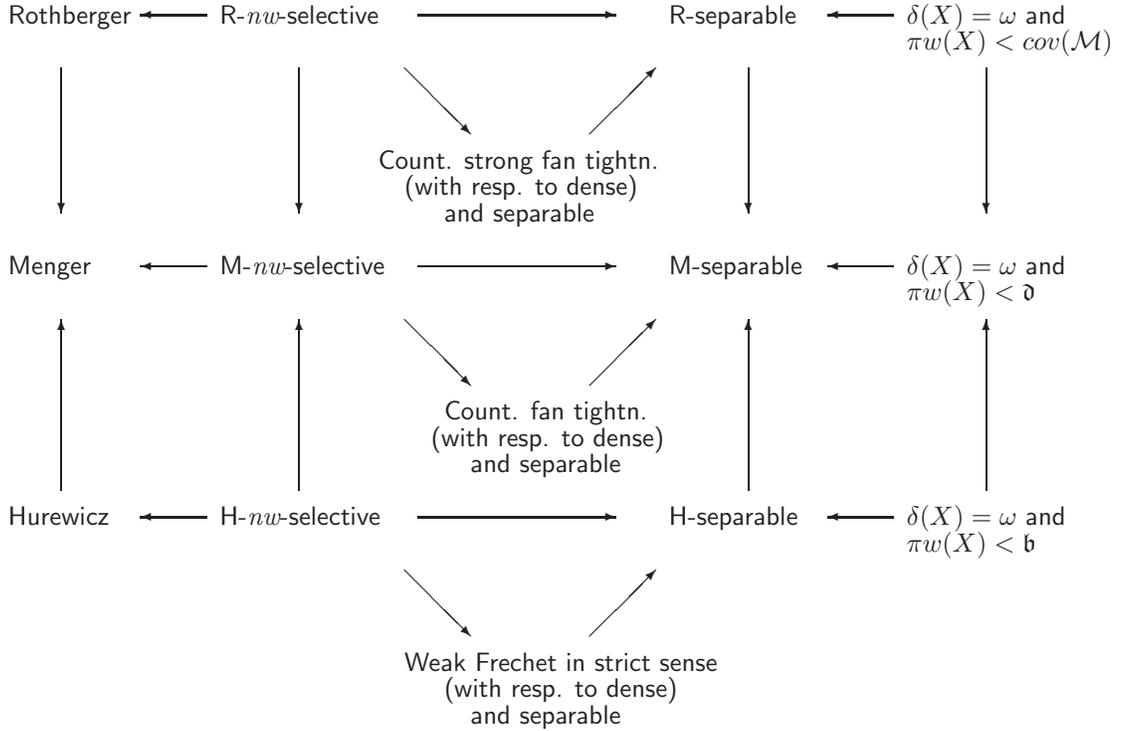
\vspace{1cm}
	\begin{flushleft}
	\begin{picture}(300,180) 
	\put(80,95){\textsf{{M-{\it nw}-selective}}} 
	\put(80,190){\textsf{{R-{\it nw}-selective}}}
	\put(80,0){\textsf{{H-{\it nw}-selective}}}
	\put(0,95){\textsf{{Menger}}} 
	\put(0,190){\textsf{{Rothberger}}}
	\put(0,0){\textsf{{Hurewicz}}}
	\put(250,95){\textsf{{M-separable}}} 
	\put(250,190){\textsf{{R-separable}}}
	\put(250,0){\textsf{{H-separable}}} 
	%\put(160,130){{\sf $$\boxed{ \begin{array}{c}
			%	\\
			%	\text{Count. strong fan tightn.}\\
			%	\text{ + separable }\\
			%	\\
			%	\end{array}
			%}$$}}
	
	\put(140,135){\textsf{{Count. strong fan tightn.}}}
	\put(150,125){\textsf{{(with resp. to dense)}}}
	\put(165,115){\textsf{{and separable}}}
	\put(165,40){\textsf{{Count. fan tightn.}}}
	\put(160,30){\textsf{{(with resp. to dense)}}}
	\put(175,20){\textsf{{and separable}}}
	\put(150,-55){\textsf{{Weak Frechet in strict sense}}}
	\put(165,-65){\textsf{{(with resp. to dense)}}}
	\put(175,-75){\textsf{{and separable}}}
	\put(155,98){\vector(1,0){75}}
	\put(155,193){\vector(1,0){75}}
	\put(155,3){\vector(1,0){75}}
	\put(110,173){\vector(0,-1){55}}
	\put(110,13){\vector(0,1){65}}
	\put(75,98){\vector(-1,0){25}}
	\put(75,193){\vector(-1,0){25}}
	\put(75,3){\vector(-1,0){25}}
	\put(20,173){\vector(0,-1){55}}
	\put(20,13){\vector(0,1){65}}
	\put(280,173){\vector(0,-1){55}}
	\put(280,13){\vector(0,1){65}}
	\put(150,173){\vector(1,-1){25}}
	\put(220,148){\vector(1,1){25}}
	\put(150,78){\vector(1,-1){25}}
	\put(220,53){\vector(1,1){25}}
	\put(150,-17){\vector(1,-1){25}}
	\put(220,-42){\vector(1,1){25}}
	\put(340,190){\textsf{{$\delta(X)=\omega$ and}}}
	\put(340,180){\textsf{{$\pi w(X)<cov({\mathcal M})$}}}
	\put(340,95){\textsf{{$\delta(X)=\omega$ and}}}
	\put(340,85){\textsf{{$\pi w(X)<{\hot d}$}}}
	\put(340,0){\textsf{{$\delta(X)=\omega$ and}}}
	\put(340,-10){\textsf{{$\pi w(X)<{\hot b}$}}}
	\put(335,193){\vector(-1,0){25}}
	\put(335,98){\vector(-1,0){25}}
	\put(335,3){\vector(-1,0){25}}
	\put(370,173){\vector(0,-1){55}}
	\put(370,13){\vector(0,1){65}}
	\end{picture}	
\end{flushleft}
\vspace{3cm}
Among other things, in \cite{BG} the following question was posed.
\begin{question}\label{q1}
	Are there uncountable M-{\it nw}-selective (R-{\it nw}-selective or H-{\it nw}-selective) spaces?
\end{question}

In Section \ref{1} we give consistent answer to Question \ref{q1} and we define \textit{trivial} R-, H-, and M-{\it nw}-selective spaces as the ones with countable net weight having, additionally, the cardinality and the weight strictly less then $cov(\M)$, $\frak b$, and $\frak d$, respectively. Since we establish that spaces having cardinalities at least $cov(\M)$, $\frak b$, and $\frak d$, fail to have the R-, H-, and M-{\it nw}-selective properties, respectively, non-trivial examples should eventually have weight greater than or equal to the respective small cardinals.

In Section \ref{sec:non-triv_suff}, using forcing methods, we construct consistent countable non-trivial examples of R-{\it nw}-selective and H-{\it nw}-selective spaces.

In Section \ref{sec:non-triv_necc} we establish some limitations to constructions of non-trivial examples.

In Section \ref{sec:prod} we consistently prove the existence of two H-{\it nw}-selective spaces whose product fails to be M-{\it nw}-selective.
 
Finally, in Section \ref{sec:HFD} we study the relations between a strong version of the HFD property (see Definition \ref{verystrongHFD}) and the R-{\it nw}-selectivity.
%\smallskip

%\section{Terminologies and preliminaries}
%A family $\mathcal N$ of sets is called a network for $X$ if for every $x\in X$ and for every open neighbourhood $U$ of $x$ there exists an element $N$ of $\mathcal N$ such that $x\in N\subseteq U$; $nw(X)=min\{|{\mathcal N}|: {\mathcal N}$ is a network for $X\}$ is the net weight of $X$. 
%$iw(X)=min\{w(Y): Y \hbox{ is the continuous bijective image of } X\}$, where $w(X)$ denotes the weight of the space $X$, is the injective weight of $X$. It is known that $iw(X)\leq nw(X)\leq w(X)$, and in the class of compact Hausdorff spaces $iw(X)= nw(X)= w(X)$ (see \cite{Ba}).

%%%%%%%%%%%%%%%%%%%%%%%%%%%%%%%%%%%%%%%%%%%%%%%%%%%%%%%%%%
%\begin{theorem}\label{thm:miller_triv}
%In the Miller model every countable $M$-$nw$-selective space $X$ has
%weight $w(X)<\hot d$.
%\end{theorem}

%\begin{theorem}\label{thm:laver_triv}
%In the Laver model every countable $H$-$nw$-selective Hausdorff space $X$ has
%weight $w(X)<\hot b$.
%\end{theorem}

%%%%%%%%%%%%%%%%%%%%%%%%%%%%%%%%%%%%%%%%%%%%%%%%%%%%%%%%%%%%

\section{Cardinality and weight of M-$\mathit{nw}$-, R-$\mathit{nw}$- and H-$\mathit{nw}$-selective spaces}\label{1}

\begin{proposition}\label{prp:m-card}%\rm 
	Let $X$ be a space such that $nw(X)=\omega$,  $|X|<{\hot d}$ and $w(X)<{\hot d}$. Then $X$ is M-{\it nw}-selective.	
\end{proposition}
\begin{proof}
	Let $\kappa,\lambda<{\hot d}$ be two cardinals such that $X=\{x_\alpha: \alpha<\kappa\}$ and $\mathcal B=\{B_\beta:\beta<\lambda\}$ is a  base for $X$.
	Let $\la \mathcal N_n:n\in \omega\ra$ be a sequence of countable networks on $X$, where $\mathcal N_n=\{N_m^n: m\in\omega\}$.
	For every $\alpha <\kappa$ and $\beta<\lambda$ such that $x_\alpha\in B_\beta$ consider the function $f_{\alpha,\beta}\in\omega^\omega$ defined by $f_{\alpha,\beta}(n)=\min\{m: x_\alpha\in N_m^n\subseteq B_\beta\}$. The family $\{f_{\alpha,\beta}: \alpha<\kappa,\beta<\lambda, x_\alpha\in B_\beta\}$ is not dominating, and hence  there exists a function $f\in\omega^\omega$ such that $f\not\leq^* f_{\alpha,\beta}$ for every $\alpha <\kappa$ and $\beta<\lambda$ such that 
 $x_\alpha\in B_\beta$. A direct verification shows that $\{N_m^n: n\in \omega,  m\leq f(n)\}$ is a network for $X$.
\end{proof}

In an analogous way it is possible to prove the following two propositions, for the first one using
Theorem~\ref{thm:cov-char}.
\begin{proposition} \label{prp:r-card}%\rm
	Let $X$ be a space such that $nw(X)=\omega$, $|X|<cov({\mathcal M})$ and $w(X)<cov(\mathcal M)$. Then $X$ is R-{\it nw}-selective.	
\end{proposition}

\begin{proposition}\label{prp:h-card}%\rm
	Let $X$ be a space such that $nw(X)=\omega$, $|X|<{\hot b}$ and $w(X)<{\hot b}$. Then $X$ is H-{\it nw}-selective.	
\end{proposition}

%\my{Propositions \ref{prp:m-card}, \ref{prp:r-card}, and \ref{prp:h-card} imply that the positive answer to Question \ref{q1} is consistent: If $\omega_1< {\hot d}$ (resp. $\omega_1< cov({\mathcal M})$, $\omega_1< \hot b$), then any $X\in [\mathbb R]^{\w_1}$ is $M$-$\mathit{nw}$-selective (resp. $R$-$\mathit{nw}$-selective, $H$-$\mathit{nw}$-selective).}

In what follows we shall call spaces satisfying the assumptions of
Proposition~\ref{prp:m-card}, \ref{prp:h-card} and \ref{prp:r-card}
\emph{trivial examples} of M-{\it nw}-selective, H-{\it nw}-selective\ and R-{\it nw}-selective\ spaces, respectively.
The reason for this terminology is that such spaces have these properties 
solely due to cardinality considerations, and not because of some 
specific structure etc.

\begin{proposition}\label{prop:not_d}%\rm
	Let $X$ be a space such that $|X|\geq\hot d$. Then $X$ is not M-{\it nw}-selective.
\end{proposition}
\begin{proof}
	 Suppose that $\mathit{nw}(X)=\w$, pick $Y\subseteq X$ such that $|Y|={\hot d}$ and let $h:Y\to D$ be a bijection, where $D\subseteq \omega^\omega$ is dominating. Without loss of generality
   we can assume that for every $f\in \w^\w$ there exists $g\in D$ such that 
   $f(n)\leq g(n)$ for all $n\in\w$ (in what follows we shall write $f\leq g$ in such cases).
   For any $n,k\in\omega$ put $Y_k^n=\{y\in Y:h(y)(n)=k\}$ and consider the countable cover ${\mathcal A}_n=\{Y_k^n: k\in\omega\}\cup \{X\setminus Y\}$ of $X$. For any $n\in\omega$ and  every ${\mathcal B}_n\in[{\mathcal A}_n]^{<\omega}$, $\bigcup_{n\in\omega}\bigcup{\mathcal B}_n\not\supseteq Y$. Indeed, pick $g\in\omega^\omega$ such that ${\mathcal B}_n\subseteq \{Y_k^n:k\leq g(n)\}\cup \{X\setminus Y\}$. Pick $y\in Y$ such that $h(y)(n)>g(n)$ for every $n\in \omega$. Then $y\not\in Y_k^n$ for every $k\leq g(n)$, so $y\not\in\bigcup_{n\in\omega}\bigcup {\mathcal B}_n$. Let ${\mathcal N}$ be a countable network of $X$. For each $n\in\omega$ consider the networks $${\mathcal N}_n={\mathcal N}\wedge {\mathcal A}_n=\{N\cap A: N\in{\mathcal N}, A\in {\mathcal A}_n\}.$$
    It follows that if   ${\mathcal F}_n\in [{\mathcal N}_n]^{<\omega}$ for each $n\in \omega$,  then $X$ is not covered by the family $\bigcup_{n\in\omega}{\mathcal F}_n$, so it cannot be a network for
    $X$.
\end{proof}

\begin{corollary}\label{cor:m-card} %\rm
	Let $X$ be a space with $nw(X)=\omega$ and $w(X)<{\hot d}$. Then  $X$ is M-{\it nw}-selective iff
		 $|X|<{\hot d}$.
	In particular, this equivalence holds for  metrizable separable spaces. 
\end{corollary}

%\my{\begin{example}\label{ex:nonM}
%	A non M-{\it nw}-selective space $X$ such that $nw(X)=\omega$,  $|X|={\frak d}$, and $w(X)={\frak d}$.
%\end{example}}
%\my{Let $X$ be a dense subset of $C_p(Y)$ of cardinality $\frak d$, where $Y\subseteq \omega^\omega$ and $|Y|={\frak d}$. Clearly, $nw(X)\leq \omega$ and $\chi(X)=\chi({\Bbb R}^X)=|X|={\frak d}$, then in this case $w(X)=\frak d$. By Proposition \ref{prop:not_d}, $X$ is not M-{\it nw}-selective. $\hfill\triangle$}
\bigskip

By Propositions \ref{prp:m-card} and \ref{prop:not_d}, it is possible to formulate the following result.
\begin{corollary}\label{prop:eqd}
	The following are equivalent facts.
	\begin{enumerate}
		\item $\omega_1<{\frak d}$;
		\item Every space $X$ with $|X|=\omega_1$, $w(X)=\omega_1$, and $nw(X)=\omega$, is M-{\it nw}-selective.
	\end{enumerate}
\end{corollary}

The following fact is analogous to Proposition~\ref{prop:not_d}, we 
present its proof for the sake of completeness.

\begin{proposition}\label{prop:not_covM}%\rm
	Let $X$ be a topological space such that $|X|\geq cov({\mathcal M})$. 
 Then $X$ is not R-{\it nw}-selective.
\end{proposition}
\begin{proof}
Suppose that $\mathit{nw}(X)=\w$, pick $Y\subseteq X$ such that $|Y|=cov(\mathcal M)$ and let $h:Y\to F'$ be a bijection, where $F'\subseteq \omega^\omega$ 
is such that for every $g\in\w^\w$ there exists $f\in F'$ such that $g(n)\neq f(n)$ for all $n\in\w$. 
Such an $F'$ exists, e.g., we could take $F$ satisfying Theorem~\ref{thm:cov-char} and set 
$F'=\{z\in\w^\w:\exists f\in F \,(z=^*f)\}$.
    For any $n,k\in\omega$ put $Y_k^n=\{y\in Y:h(y)(n)=k\}$ and consider the countable cover ${\mathcal A}_n=\{Y_k^n: k\in\omega\}\cup \{X\setminus Y\}$ of $X$. For any
    $g\in\w^\w$ we have  $\bigcup_{n\in\omega} Y^n_{g(n)}\not\supseteq Y$. Indeed, pick $f\in F'$ with $f(n)\neq g(n)$ for all $n\in\w$, and let 
    $y\in Y$ be such that $h(y)=f$. Then $y\not\in Y_{g(n)}^n$ for every $n\in\w$ because
    $y\in Y_{f(n)}^n$ and $Y_{f(n)}^n\cap Y_{g(n)}^n=\emptyset$, since the family $\{Y^n_k:k\in\w\}$ is disjoint by the definition. 
    Let ${\mathcal N}$ be a countable network of $X$. For each $n\in\omega$ consider the networks 
    $${\mathcal N}_n={\mathcal N}\wedge {\mathcal A}_n=\{N\cap A: N\in{\mathcal N}, A\in {\mathcal A}_n\}.$$
    It follows that if   $N_n\in {\mathcal N}_n$ for each $n\in \omega$,  then $Y$ is not covered by the family $\{N_n:n\in\w\}$, so it cannot be a network for
    $X$.
\end{proof}

\begin{corollary}\label{cor:r-card} %\rm
	Let $X$ be a space with $nw(X)=\omega$ and $w(X)<cov(\M)$. Then  $X$ is R-{\it nw}-selective iff
		 $|X|<cov(\M)$.
	In particular, this equivalence holds for  metrizable separable spaces. 
\end{corollary}

%\my{The following example is the analogous to the one of Example \ref{ex:nonM}. In fact, it is given without construction.}
%\my{\begin{example}\label{ex:nonR}
%		A non R-{\it nw}-selective space $X$ such that $nw(X)=\omega$,  $|X|=cov({\M})$, and $w(X)=cov({\M})$.
%\end{example}}

%\my{Let $X$ be a dense subset of $C_p(Y)$ of cardinality $cov({\cal M})$, where $Y\subseteq \omega^\omega$ and $|Y|=cov({\cal M})$. Clearly, $nw(X)\leq \omega$ and $\chi(X)=\chi({\Bbb R}^X)=|X|=cov({\cal M})$, then in this case $w(X)=cov({\cal M})$. By Proposition \ref{prop:not_r}, $X$ is not R-{\it nw}-selective. $\hfill\triangle$}

By Propositions \ref{prp:r-card} and \ref{prop:not_covM}, it is possible to formulate the following result.
\begin{corollary}\label{prop:eqcov}
		The following are equivalent facts.
		\begin{enumerate}
			\item $\omega_1<cov({\M})$;
			\item Every space $X$ with $|X|=\omega_1$, $w(X)=\omega_1$, and $nw(X)=\omega$, is R-{\it nw}-selective.
		\end{enumerate}
\end{corollary}

By Propositions \ref{prop:eqd} and \ref{prop:eqcov}, if $\omega_1=cov(\M)<{\frak d}$ holds, each space of cardinality and weight equal to $\omega_1$ and countable netweight is M-{\it nw}-selective not R-{\it nw}-selective.\\
%%%%%%%%%%%%%%%%%%%%%%%%%%%%%%%%%%%%%%%%%%%%%%%%%%%%%%%%%%%%%%

As in the case of Proposition~\ref{prop:not_covM}, 
the next fact is also analogous to Proposition~\ref{prop:not_d}, but we nonetheless 
present its proof for the sake of completeness.

\begin{proposition}\label{prop:not_b}%\rm
	Let $X$ be a space such that $|X|\geq\hot b$. Then $X$ is not H-{\it nw}-selective.
\end{proposition}
\begin{proof}
	 Suppose that $\mathit{nw}(X)=\w$, pick $Y\subseteq X$ such that $|Y|={\hot b}$ and let $h:Y\to B$ be a bijection, where $B\subseteq \omega^\omega$ is unbounded. Let $Y_k^n$ and  ${\mathcal A}_n$
  be defined in the same way as in the proof of Proposition~\ref{prop:not_d}. For any $n\in\omega$ and  every ${\mathcal B}_n\in[{\mathcal A}_n]^{<\omega}$
  there exists $I\in [\w]^\w$ with
  $\bigcup_{n\in I}\cup{\mathcal B}_n\not\supseteq Y$. Indeed, pick $g\in\omega^\omega$ such that ${\mathcal B}_n\subseteq \{Y_k^n:k\leq g(n)\}\cup \{X\setminus Y\}$. Pick $y\in Y$ such that $h(y)(n)>g(n)$ for 
  infinitely many
  $n\in \omega$, and let $I$ be the set of all such $n$. 
    %%%%%%%%%%%%%%%%%%%%%%%%%%%%%%%%%%%%%%%%%%%%%
  Then $y\not\in Y_k^n$ for every $k\leq g(n)$ and $n\in I$, so $y\not\in\bigcup_{n\in I}\bigcup {\mathcal B}_n$. Let ${\mathcal N}$ be a countable network of $X$. For each $n\in\omega$ consider the networks ${\mathcal N}_n$
  defined in the same way as in the proof of Proposition~\ref{prop:not_d} and note that if ${\mathcal F}_n\in [{\mathcal N}_n]^{<\omega}$ for each $n\in \omega$,  then $y$ is not covered by the family $\bigcup_{n\in I}{\mathcal F}_n$. Thus, for every  $m\in\w$ there exists $n\geq m$ (namely $\min (I\setminus m)$) such that no $F\in\mathcal F_n$ contains $y$, which implies that $X$ is not H-{\it nw}-selective. 
\end{proof}

\begin{corollary}\label{cor:h-card} %\rm
	Let $X$ be a space with $nw(X)=\omega$ and $w(X)<{\hot b}$. Then  $X$ is H-{\it nw}-selective iff
		 $|X|<{\hot b}$.
	In particular, this equivalence holds for  metrizable separable spaces. 
\end{corollary}

%\my{The following example is the analogous to the one of Example \ref{ex:nonM}. In fact, as the Example \ref{ex:nonR} is given without construction.}
%\my{\begin{example}\label{ex:nonH}
%		A non H-{\it nw}-selective space $X$ such that $nw(X)=\omega$,  $|X|={\frak b}$, and $w(X)={\frak b}$.
%\end{example}}

%\my{Let $X$ be a dense subset of $C_p(Y)$ of cardinality ${\frak b}$, where $Y\subseteq \omega^\omega$ and $|Y|={\frak b}$. Clearly, $nw(X)\leq \omega$ and $\chi(X)=\chi({\Bbb R}^X)=|X|={\frak b}$, then in this case $w(X)={\frak b}$. By Proposition \ref{prop:not_h}, $X$ is not H-{\it nw}-selective. $\hfill\triangle$}

By Propositions \ref{prp:h-card} and \ref{prop:not_b} it is possible to formulate the following result.
\begin{corollary}\label{prop:eqb}
		The following are equivalent facts.
		\begin{enumerate}
			\item $\omega_1<{\frak b}$;
			\item Every space $X$ with $|X|=\omega_1$, $w(X)=\omega_1$, and $nw(X)=\omega$, is H-{\it nw}-selective.
		\end{enumerate}
\end{corollary}
By Propositions \ref{prop:eqd} and \ref{prop:eqb}, if $\omega_1={\frak b}<{\frak d}$ holds, each space of cardinality and weight equal to $\omega_1$ and countable netweight is M-{\it nw}-selective not H-{\it nw}-selective. By Propositions \ref{prop:eqcov} and \ref{prop:eqb}, if $\omega_1=cov(\M)<{\frak b}$ holds, each space of cardinality and weight equal to $\omega_1$ and countable netweight is H-{\it nw}-selective not R-{\it nw}-selective. By Propositions \ref{prop:eqcov} and \ref{prop:eqb}, if $\omega_1={\frak b}<cov(\M)$ holds, each space of cardinality and weight equal to $\omega_1$ and countable netweight is R-{\it nw}-selective not H-{\it nw}-selective.
Additionally, Corollaries \ref{cor:m-card}, \ref{cor:r-card} and \ref{cor:h-card}
allow us to find consistent examples of sets of reals   distinguishing 
between the corresponding properties: 
If $cov({\mathcal M})<{\hot d} $
(resp. ${\hot b}<{\hot d}$, $cov({\mathcal M})<{\hot b}$, $\hot b<cov({\mathcal M})$), then any $X\in [\mathbb R]^{cov({\mathcal M})}$
(resp. $X\in [\mathbb R]^{\hot b}$, $X\in [\mathbb R]^{cov({\mathcal M})}$, $X\in [\mathbb R]^{\hot b}$)
is M-{\it nw}-selective\ but not R-{\it nw}-selective\ (resp. M-{\it nw}-selective\ but not H-{\it nw}-selective, H-{\it nw}-selective\ but not R-{\it nw}-selective,
R-{\it nw}-selective\ but not H-{\it nw}-selective). 
  
However, we do not know of any examples in ZFC distinguishing these properties, because at the moment it is not even known whether there are in ZFC
non-trivial countable spaces which are M-{\it nw}-selective, H-{\it nw}-selective, or R-{\it nw}-selective.
More precisely, the following question 
(in fact, each of  the 6 subquestions it naturally comprises)
is still open.

\begin{question} \label{q2}
	\begin{enumerate}
	    \item Is there a ZFC example of a non-trivial M-{\it nw}-selective\  (resp. R-{\it nw}-selective, H-{\it nw}-selective) space $X$? 
 What about countable  spaces?
\item Is the existence  of a non-trivial uncountable M-{\it nw}-selective\  (resp. R-{\it nw}-selective, H-{\it nw}-selective) space $X$ consistent with ZFC?
\end{enumerate}
\end{question}

%%%%%%%%%%%%%%%%%%%%%%%%%%%%%%%%%%%%%%%%%%%%%%%%%
%%%%%%%%%%%%%%%%%%%%%%%%%%%%%%%%%%%%%%%%%%%%%%
%%%%%%%%%%%%%%%%%%%%%%%%%%%%%%%%%%%%%%%%%%%%%%%%%%%%%%%%%%%%%%%%%

\section{$\mathit{nw}$-Selectivity of countable  subspaces of $C_p(X,2)$: sufficient conditions and consistent non-trivial examples.} \label{sec:non-triv_suff}

For a topological space $X$ we denote by
\begin{itemize}
\item $\B(X)$  the family of all
countable  Borel   covers of $X$;
\item $\B_\Omega(X)$  the family of all
countable  Borel   $\w$-covers of $X$;
\item $\B_\Gamma(X)$  the family of all
countable  Borel   $\gamma$-covers of $X$.
\end{itemize}
We omit $X$ from these notations if it is clear from the context.

%%%%%%%%%%%%%%%%%%%%%%%%%%%%%%%%%%%%%%%%%%%%%%%%%%%%%%%%%%%%%%%%%

\begin{theorem} \label{thm:suff_rnw}
	Suppose that $X\subseteq 2^\omega$ is such that $X^n$ is $S_1(\B,\B)$ for every $n\in\omega$. 
  Let $Y\subseteq C_p(X,2)$ be a countable  subset. Then $Y$ is R-{\it nw}-selective. \ Moreover, if 
$Y$ is  
  dense, then
  $w(Y)=|X|$.
\end{theorem}
\begin{proof}
	 Clearly,  $|X|\geq w(Y)\geq \chi(Y)$, and if $Y$ is dense, then additionally we have\footnote{This part does not use any additional properties of $X$ like 
  $S_1(\B,\B)$.}
  $\chi(Y)= \chi(C_p(X,2))=|X|$,
  so in this case $w(Y)=|X|$. 

     Let $\{y_j:j\in\omega\}$ be an enumeration of $Y$ and ${\mathcal N}_k=\{N_m^k:m\in\omega\}\subset\mathcal P(Y)$  a countable network for each $k\in\omega$. Given a basic open subset of $C_p(X,2)$ of the form 
     $$[\vec{x},\vec{\epsilon}]=\{f\in C_p(X,2) : f(x_0)=\epsilon_0,...,f(x_{n-1})=\epsilon_{n-1}\},$$
     where $\vec{x}\in X^n,\vec{\epsilon}\in 2^n$, and $j\in\omega$, set
     $A_{j,\vec{\epsilon}}=\{\vec{x}\in X^n:y_j\in [\vec{x},\vec{\epsilon}]\}.$
Let  
     $h_{j,\vec{\epsilon}}: A_{j,\vec{\epsilon}}\to \omega^\omega$ be a function defined by 
     $$h_{j,\vec{\epsilon}}(\vec{x})(k)=\min\{m:y_j\in N_m^k\subseteq [\vec{x},\vec{\epsilon}]\}.$$
$A_{j,\vec{\epsilon}}$ satisfies $S_1(\B,\B)$ because this property is hereditary
by \cite[Theorem~13]{SchTsa02}, and therefore 
$h_{j,\vec{\epsilon}}[A_{j,\vec{\epsilon}}]$ is Rothberger by \cite[Theorem~14]{SchTsa02} because the function $h_{j,\vec{\epsilon}}$ is clearly  Borel. Then $R:=\bigcup_{n\in\omega, j\in\omega,\vec{\epsilon}\in 2^n}h_{j, \vec{\epsilon}}[A_{j,\vec{\epsilon}}]$ is Rothberger, being a countable union of Rothberger spaces. So there exists $h\in\omega^\omega$ such that 
     $\{k\in\w: r(k)=h(k)\}$ is infinite for every $r\in R$.
      As a result. $\{N_{h(k)}^k: k\in\omega\}$ is a network for $Y$. 
      Indeed, pick a basic open set $[\vec{x},\vec{\epsilon}]$ and a point $y_j \in [\vec{x},\vec{\epsilon}]$.  Then $\vec{x}\in A_{j,\vec{\epsilon}}$, and therefore  there exists $k\in\omega$ such that $h_{j,\vec{\epsilon}}(\vec{x})(k) = h(k)$, hence
      $$y_j\in N^k_{h_{j,\vec{\epsilon}}(\vec{x})(k)}=N^k_{h(k)}\subseteq [\vec{x},\vec{\epsilon}],$$
which completes our proof.
\end{proof}

One of the ways to get non-trivial (namely those having size at least $\mathit{cov}(\M)$)
examples of spaces $X$ like in Theorem~\ref{thm:suff_rnw} is using forcing. 
This approach is not new and in a slightly different form was used in 
\cite{Bre96}, so the next fact may be thought of as folklore. We present its proof 
since we were unable to find it published elsewhere in the form we need. 

\begin{proposition}    \label{prop: Bre}
	Let $X=\{c_\alpha: \alpha<\lambda\}$ be the  set of Cohen generic reals over the ground model $V$
added by the standard poset $\mathit{Fin}(\lambda\times \omega, 2)$ consisting of finite partial functions 
from $\lambda\times\w$ to $2$, where $\lambda$ is an uncountable cardinal. Let $G$ be $\mathit{Fin}(\lambda\times \omega, 2)$-generic
filter giving rise to $X$.
 Then in $V[G]$, for any $k\in\omega$ and a sequence $\langle {\mathcal B}_n: n\in\omega\rangle $ of countable Borel covers of $X^k$, for each $n\in\omega$ there is $B_n\in{\mathcal B}_n$ such that $X^k\subseteq \bigcup_{n\in\omega}B_n$. I.e., all finite power of $X$
 satisfy $S_1(\B,\B)$.
 \end{proposition}

 We shall need the following standard fact whose proof we add for the sake of 
 completeness.

\begin{lemma} \label{lem:coh_in_nonmeag}
Let $\lambda, G, X$ be such as in Proposition~\ref{prop: Bre} and
suppose that $D\subset 2^\omega$ is a Borel non-meager set coded in the ground model $V$. Then there exists $\beta<\lambda $ such that $c_\beta\in D$.
\end{lemma}
\begin{proof}
 	Since $D$ is non-meager, there exist $s\in 2^{<\omega}$ such that $D\cap [s]$ is comeager in $[s]$, i.e., $[s]\setminus D$ is meager.
 (Recall that $[s]=\{z\in 2^\w:z\uhr|s|=s\}$.) Fix $p\in Fin(\lambda\times \omega, 2)$ and  $\beta\in\lambda $ such that $\mathit{dom}(p)\cap
 (\{\beta\}\times \omega)=\emptyset$. Let $q=p\cup r$, where $dom(r)=\{\beta\}\times |s|$ and $r(\beta,j)=s(j)$ for every $j\in |s|$. 
 Then 
 $$q\Vdash c_\beta \in [s] \wedge  c_\beta  \mbox{ lies in every comeager set coded in } V,$$
and hence  $q$ also forces $ c_\beta \in D\cap [s]$.
Now the statement of the lemma follows by the density argument. 
\end{proof}
\smallskip

\noindent\textit{Proof of Proposition~\ref{prop: Bre}.}   
	We will prove it by induction on $k$. For $k=0$ there is nothing to prove. Assuming that it is true for any natural number  $\leq k$, we will prove our statement for $k+1$. Consider ${\mathcal B}_n=\{B_i^n:i\in\omega\}\in\B(X^{k+1})$, for every $n\in\omega$. Let $A\in[\lambda]^\omega$ be such that $\langle \langle B^n_i: i\in\omega \rangle:n\in\omega\rangle$ is coded in $V[\{c_\alpha:\alpha\in A\}]$. Let $\omega=I_0\sqcup I_1$ be a partition into two infinite disjoint sets. The set $(2^\omega)^{k+1}\setminus \bigcup {\mathcal B}_n$ is meager in $(2^\omega)^{k+1}$ for every $n\in\omega$. Indeed, suppose  that contrary to our claim there exists $n\in\omega$ such that $K:=(2^\omega)^{k+1}\setminus \bigcup {\mathcal B}_n$ is 
non-meager. Then Lemma~\ref{lem:coh_in_nonmeag} implies that there exists 
an injective sequence $\la \beta_i:i<k+1\ra$ of ordinals in $\lambda\setminus A$
such that $\la c_{\beta_i}:i<k+1\ra\in K$, which is impossible because 
$\B_n$ covers $X^{k+1}$.

%%%%%%%%%%%%%%%%%%%%%%%%%%%%%%%%%%%%%%%%%%%%%%%%%%%%%%%%%%%%%%%%%

	For every $n\in I_0$ pick $i_n\in\omega$ such that $\bigcup_{n\in I_0} B^n_{i_n}$ is comeager in $(2^\omega)^{k+1}$. This could be done in
 $V[\{c_\alpha:\alpha\in A\}]$ as follows: Given   an enumeration $\{\vec{s}_n:n\in I_0\}$ of $(2^{<\omega})^{k+1}$, let $i_n$ be such that $B^n_{i_n}\cap [\vec{s}_n]$ is non-meager in $[\vec{s}_n]$, $n\in I_0$, where $[\vec{s}_n]=\prod_{j\leq k}[s_j^n]$. Then the union $\bigcup_{n\in I_0} B^n_{i_n}$ is comeager in $(2^\omega)^{k+1}$,
 because its intersection with each clopen subset of $(2^\omega)^{k+1}$ is non-meager. Fix any mutually different $\alpha_0,...,\alpha_k\in \lambda\setminus A$. Then $\langle c_{\alpha_0},..., c_{\alpha_k}\rangle\in \bigcup_{n\in I_0}B^n_{i_n}$ because any such $\langle c_{\alpha_0},..., c_{\alpha_k}\rangle$ lies in every comeager subset of $(2^\omega)^{k+1}$ coded in $V[c_\alpha: \alpha\in A]$. From the above we conclude that 
 $$Y:=X^{k+1}\setminus \bigcup_{n\in I_0} B^n_{i_n}\subset \big\{\langle c_{\alpha_0},..., c_{\alpha_k}\rangle: \exists j\leq k \: (\alpha_j\in A) \vee \exists j_1,j_2\leq k\: (\alpha_{j_1}=\alpha_{j_2})\big\}.$$
Thus $Y$ is covered by a  countable union of homeomorphic copies of $X^j$ with $j\leq k$, 
hence by our assumption we can conclude the proof by covering $Y$ with 
suitably chosen $B^n_{i_n}$'s for $n\in I_1$. 
\hfill $\Box$
\medskip

Combining the results above and the fact that 
$\mathit{cov}(\M)=\hot d=\hot c=\lambda $ after adding $\lambda$-many Cohen reals to a ground model of GCH, where 
$\lambda$ is a  cardinal of uncountable cofinality, we get a consistent non-trivial example of
a R-{\it nw}-selective\ space which is also a non-trivial example of
a M-{\it nw}-selective\ space.

%%%%%%%%%%%%%%%%%%%%%%%%%%%%%%%%%%%%%%%%%%%%%%%%%

\begin{corollary} \label{cor:cohens}
	Suppose that   GCH holds in the ground model $V$. Let
$\lambda$ be a  cardinal of uncountable cofinality and   $G,X$   such as in Proposition~\ref{prop: Bre}.
 Finally, in $V[G]$ let $Y\subset C_p(X,2)$ be a countable dense subspace. Then 
 $Y$ is a R-{\it nw}-selective\ (and hence M-{\it nw}-selective) and $w(Y)=\mathit{cov}(\M)=\hot d$. 
\end{corollary}

The above corollary motivates the following question, which is related to Question~\ref{q2}.

\begin{question} \label{q3}
	Is there a consistent example of a countable R-{\it nw}-selective\ space of weight $>\mathit{cov}(\M)$?
\end{question}

%%%%%%%%%%%%%%%%%%%%%%%%%%%%%%%%%%%%%%%%%%%%%%%%%%%%%%%%%%%%%%%%%%%%%%%%%%%%%%%%%%%
%%

Theorem~\ref{thm:suff_rnw}, Proposition~\ref{prop: Bre} and Lemma~\ref{lem:coh_in_nonmeag}
have their counterparts for random reals, with ``Cohen''  and ``meager''
replaced with ``random'' and ``measure zero'', respectively.
We omit proofs which are completely analogous, i.e., those of Theorem~\ref{thm:suff_hnw} and Lemma~\ref{lem:ran_in_nonzero}.

Again, this approach of using random reals could be traced back in some sense to
\cite{Bre96} and hence  may be thought of as folklore. We refer 
the reader to \cite[Section~3.1]{BJ} for more information about the random forcing.

\begin{theorem} \label{thm:suff_hnw}
	Suppose that $X\subseteq 2^\omega$ is such that $X^n$ is $S_1(\B_\Gamma,\B_\Gamma)$ for every $n\in\omega$. 
  Let $Y\subseteq C_p(X,2)$ be a countable  subset. Then $Y$ is H-{\it nw}-selective. \ Moreover, if 
$Y$ is  
  dense,
  $w(Y)=|X|$.
\end{theorem}

\begin{proposition}    \label{prop: Bre_random}
	Let $X=\{r_\alpha: \alpha<\lambda\}$ be the  set of  generic random reals over the ground model $V$
added by the standard poset $B(\lambda)=\mathit{Bor}(2^{\lambda\times\w})/\mathcal Z_\lambda$, where 
$\mathcal Z_\lambda$ is the ideal of subsets of $2^{\lambda\times\w}$ which have measure $0$ with respect to the 
usual product probability measure on $2^{\lambda\times\w}$. Let also 
$G$ be $B(\lambda)$-generic over $V$ giving rise to $X$.
 Then in $V[G]$, for any $k\in\omega$ and a sequence $\langle {\mathcal B}_n: n\in\omega\rangle $ of  Borel $\gamma$-covers of $X^k$, for each $n\in\omega$ there is $B_n\in{\mathcal B}_n$ such that 
 $\{B_n:n\in\w\}$ is a $\gamma$-cover of
 $X^k$. I.e., all finite power of $X$
 satisfy $S_1(\B_\Gamma,\B_\Gamma)$.
 \end{proposition}

The key part of the proof of Proposition~\ref{prop: Bre_random} relies on the following

\begin{lemma} \label{lem:ran_in_nonzero}
Suppose that $D\subset 2^\omega$ is a Borel non-measure zero set coded in the ground model $V$ and $G,X$ are
such as in Proposition~\ref{prop: Bre_random}. Then there exists $\beta<\lambda $ such that $r_\beta\in D$.
\end{lemma}
\smallskip

\noindent\textit{Proof of Proposition~\ref{prop: Bre_random}.}   
	We will prove it by induction on $k$
 that $X^k$ satisfies $U_{\mathit{fin}}(\B,\B_\Gamma)$, which is equivalent to
 $S_1(\B_\Gamma,\B_\Gamma)$ by \cite[Theorem~1]{SchTsa02}.

 For $k=0$ there is nothing to prove. Assuming that it is true for any natural number  $\leq k$, we will prove our statement for $k+1$. Consider ${\mathcal B}_n=\{B_i^n:i\in\omega\}\in\B(X^{k+1})$, for every $n\in\omega$. Let $A\in[\lambda]^\omega$ be such that $\langle \langle B^n_i: i\in\omega \rangle:n\in\omega\rangle$ is coded in $V[\{r_\alpha:\alpha\in A\}]$.  The set $(2^\omega)^{k+1}\setminus \bigcup {\mathcal B}_n$ has measure $0$ in $(2^\omega)^{k+1}$ for every $n\in\omega$. Indeed, suppose  that contrary to our claim there exists $n\in\omega$ such that $L:=(2^\omega)^{k+1}\setminus \bigcup {\mathcal B}_n$ has positive measure. Then Lemma~\ref{lem:ran_in_nonzero} implies that there exists 
an injective sequence $\la \beta_i:i<k+1\ra$ of ordinals in $\lambda\setminus A$
such that $\la r_{\beta_i}:i<k+1\ra\in L$, which is impossible because 
$\B_n$ covers $X^{k+1}$.

%%%%%%%%%%%%%%%%%%%%%%%%%%%%%%%%%%%%%%%%%%%%%%%%%%%%%%%%%%%%%%%%%

	In $V[\{r_\alpha:\alpha\in A\}]$, for every $n\in \w$ pick $i_n\in\omega$ such that 
 $\nu(\bigcup_{i\leq i_n}B^n_i)\geq 1-\frac{1}{2^n}$ and note that
 $\nu(Z)=1$, where $\nu$ is the Lebesgue measure on $(2^\w)^{k+1}$ and
 $$Z=\bigcup_{m\in\w}\bigcap_{n\geq m}\bigcup_{i\leq i_n}B^n_i.$$
   Fix any mutually different $\alpha_0,...,\alpha_k\in \lambda\setminus A$. Then $\langle r_{\alpha_0},..., r_{\alpha_k}\rangle\in Z$ because any such $\langle r_{\alpha_0},..., r_{\alpha_k}\rangle$ lies in every measure $1$ subset of $(2^\omega)^{k+1}$ coded in $V[r_\alpha: \alpha\in A]$. From the above we conclude that 
 $$Y:=X^{k+1}\setminus Z\subset \big\{\langle r_{\alpha_0},..., r_{\alpha_k}\rangle: \exists j\leq k \: (\alpha_j\in A) \vee \exists j_1,j_2\leq k\: (\alpha_{j_1}=\alpha_{j_2})\big\}.$$
Thus $Y$ is covered by a  countable union of homeomorphic copies of $X^j$ with $j\leq k$, 
hence by our assumption we can find
$\la j_n:n\in\w\ra\in \w^\w$ such that $\{\bigcup_{i\leq j_n}B^n_i:n\in\w\}\in\B_\Gamma(Y)$. Since
$\{\bigcup_{i\leq i_n}B^n_i:n\in\w\}\in\B_\Gamma(Z)$ by the choice of
$\la i_n:n\in\w\ra$, we conclude that
$$\big\{\bigcup_{i\leq \max\{i_n,j_n\}}B^n_i:n\in\w\big\}\in\B_\Gamma(Z\cup Y)=\B_\Gamma(X^{k+1}),$$
which completes our proof. 
\hfill $\Box$
\medskip

Combining the results above and the fact that 
$\hot d=\w_1$ after adding $\lambda$-many random reals to a ground model of GCH, 
we get a consistent non-trivial example of
a H-{\it nw}-selective\ space which is also a non-trivial example of
a M-{\it nw}-selective\ space. 

\begin{corollary} \label{cor:randoms}
	Suppose that   GCH holds in the ground model $V$ and
$\lambda$ is  an uncountable cardinal.  Let  $G,X$  be such as in Proposition~\ref{prop: Bre_random}.
 Finally, in $V[G]$ let $Y\subset C_p(X,2)$ be a countable dense subspace. Then   $Y$ is  H-{\it nw}-selective\ (and hence M-{\it nw}-selective) and $w(Y)=\lambda=\hot c\geq\hot d=\w_1$. 
\end{corollary}

The  corollary above among others shows that the counterparts of  Question~\ref{q3} for H-{\it nw}-selective\ and M-{\it nw}-selective\ spaces
have affirmative answers, i.e., there are consistent examples of countable
H-{\it nw}-selective\ (resp.  M-{\it nw}-selective) spaces with weight $>\hot b$ (resp. $>\hot d$).

%%%%%%%%%%%%%%%%%%%%%%%%%%%%%%%%%%%%%%%%%%%%%%%%%%%%%%%%%%%%%%%%%%%%%%%%%%%%%%%%%%%%%%%%%%

\section{Various kinds of $\mathit{nw}$-selectivity of \lq\lq standard\rq\rq countable dense  subspaces of $C_p(X,2)$: necessary conditions.} \label{sec:non-triv_necc}

In this section we establish some limitations to
constructions of non-trivial examples by the
methods developed in Section~\ref{sec:non-triv_suff}.
More precisely, we consider certain specific countable dense subspaces
of $C_p(X)$ defined before Theorem~\ref{thm:ness_c_p_Hnw}, where $X$ is a metrizable separable spaces, and show
that these having $\mathit{nw}$-selectivity properties implies 
$X$ having quite strong combinatorial covering properties with respect to the family of all countable closed covers.

We start by showing that the properties we consider are equivalent to their local counterparts 
in the realm of countable spaces. We call $\mathcal N$ a \emph{network for a space $X$ at $x\in X$,} if for every neighbourhood $U\ni x$ there exists $N\in\NN$ with $x\in N\in\NN$.
Thus,  $\NN$ is a  network for $X$ if and only if it is  a network for $X$ at each $x\in X$.

\begin{definition} \label{def:loc}\rm
A space $X$ is 
\begin{itemize}
\item
\emph{locally M-{\it nw}-selective}, if for every $x\in X$ and sequence $\la\NN_m:m\in\w\ra$
of networks for $X$ at $x$, there exists a sequence
$\la \mathcal L_m:m\in\w\ra$ such that $\mathcal L_m\in [\NN_m]^{<\w}$ for all $m$
and $\bigcup_{m\in\w}\mathcal L_m$ is a networks for $X$ at $x$;
\item
\emph{locally H-{\it nw}-selective}, if  for every $x\in X$ and sequence $\la\NN_m:m\in\w\ra$
of networks for $X$ at $x$, there exists a sequence
$\la \mathcal L_m:m\in\w\ra$ such that $\mathcal L_m\in [\NN_m]^{<\w}$ for all $m$
and $\bigcup_{m\in I}\mathcal L_m$ is a networks for $X$ at $x$ for any $I\in [\w]^\w$;
\item
\emph{locally R-{\it nw}-selective}, if  for every $x\in X$ and sequence $\la\NN_m:m\in\w\ra$
of networks for $X$ at $x$, there exists a sequence
$\la N_m:m\in\w\ra$ such that $N_m\in \NN_m$ for all $m$
and $\{N_m:m\in\w\}$ is a networks for $X$ at $x$.
\end{itemize}
\end{definition}

\begin{lemma} \label{lem:loc_char}
\begin{enumerate}
\item If $X$ is locally M-{\it nw}-selective \ (resp. H-{\it nw}-selective, R-{\it nw}-selective) and $|X|=\w$, 
then it is  M-{\it nw}-selective \ (resp. H-{\it nw}-selective, R-{\it nw}-selective).
\item If $X$ is M-{\it nw}-selective \ (resp. H-{\it nw}-selective, R-{\it nw}-selective), then it is locally M-{\it nw}-selective \ (resp. H-{\it nw}-selective, R-{\it nw}-selective).
\end{enumerate}
\end{lemma}
\begin{proof}
The first item is rather obvious. 
For  the second one it suffices to note that
if $\mathcal M$ is a (countable) network for $X$ and $\NN$ is a (countable)  network for $X$ at $x\in X$, then 
$$\NN\cup\{M\setminus\{x\}:M\in\mathcal M\}$$
is a (countable) network for $X$.
\end{proof}

In what follows we shall call a sequence  $\la\U_n:n\in\w\ra $
of finite families of subsets of $X$ a \emph{$\gamma_{\mathit{fs}}$-sequence\footnote{``fs''
is the abbreviation of \emph{finite subsets}} on $X$},
if for every $F\in [X]^{<\w}$ there exists $n\in\w$ such that for all $k\geq n$ there exists 
$U\in\U_n$ containing $F$. 

For a subset $A$ of $X$ the \emph{characteristic function of $A$} is  $\chi_A:X\to 2$ such that $\chi_A(x)=0$ iff $x\in A$.
Let $X$ be a metrizable separable zero-dimensional space
and $\CS$ a base for $X$ closed under finite unions and 
complements of its elements.
Then we denote by $Y_{\CS}$ the set 
$\{\chi_S:S\in\CS\}\subset C_p(X,2)$.

\begin{theorem} \label{thm:ness_c_p_Hnw}
Let $X$ be a metrizable separable zero-dimensional space and $\CS$ a  countable clopen base of
$X$ closed under finite unions and complements. 
If $Y=Y_{\CS}$ is H-{\it nw}-selective\  as a subspace of $C_p(X,2)$, then 
for every sequence $\la\C_n:n\in\w\ra$ of countable closed $\w$-covers of $X$
there exists a $\gamma_{\mathit{fs}}$-sequence $\la\D_n:n\in\w\ra$ on $X$ such that 
$\D_n\in[\C_n]^{<\w}$.
\end{theorem}
\begin{proof}
   Note that the constant $0$ function (which we denote by $0$)
  belongs to $Y$: Given any $S\in \CS$, we have that $X\setminus S\in\CS$, and
  hence $X=S\cup (X\setminus S)\in\CS$, which yields $0=\chi_X\in Y$.
  
    For every $C\subset X$ we denote by $[C,0]$ the set $\{f\in C_p(X):f\uhr C\equiv 0\}$.
  Let $\la\C_n:n\in\w\ra$ be a sequence of countable closed $\w$-covers of $X$.
  It is easy to check  that 
  $$\NN_n:=\big\{ [C,0]\cap Y:C\in\C_n \big\}$$
  is a network for $Y$ at $0$ for every $n\in\w$. By Lemma~\ref{lem:loc_char} we know that
  $Y$ is locally H-{\it nw}-selective, and hence there exists a sequence 
  $\la\D_n:n\in\w\ra $ such that   $\D_n\in[\C_n]^{<\w}$ for all $n\in\w$
  and 
  $$\NN:=\{[C,0]\cap Y:C\in\D_n,n\in I\}$$
  is a network for $Y$ at $0$ for any infinite $I\subset\w$. 
  
  We claim that 
$\la\D_n:n\in\w\ra $ is a $\gamma_{\mathit{fs}}$-sequence of subsets of $X$. Indeed, suppose towards a contradiction, that  there exists 
$I\in[\w]^\w$ and $A\in [X]^{<\w}$ such that $A\not\subset C$ for any $C\in\bigcup_{n\in I}\D_n$. Set $O=[A,0]\cap Y$ and note that $O$ is an open neighbourhood of $0$ in $Y$.
Thus there exists $n\in I$ and $C\in \D_n$ such that $O\supset [C,0]\cap Y$.
However, there exists $x\in A\setminus C$, and hence   there exists $S\in\CS$
with\footnote{Here we use that $Y=Y_{\CS}$ and not just arbitrary countable dense subset of $C_p(X,2)$.} 
$C\subset S$ and $x\not\in S$. It follows that $\chi_S\in [C,0]\cap Y$
and $\chi_S\not\in [A,0]\cap Y=O$, which gives the desired contradiction.  
\end{proof}

For a topological space $X$ we make the following notation:
\begin{itemize}
\item $\C(X)$ is the family of all
countable  closed   covers of $X$;
\item $\C_\Omega(X)$ is the family of all
countable  closed   $\w$-covers of $X$;
\item $\C_\Gamma(X)$ is the family of all
countable  closed   $\gamma$-covers of $X$;
\item $\C^o(X)$ is the family of all
countable  clopen   covers of $X$;
\item $\C^o_\Omega(X)$ is the family of all
countable  clopen   $\w$-covers of $X$;
\item $\C^o_\Gamma(X)$ is the family of all
countable  clopen   $\gamma$-covers of $X$.
\end{itemize}

 Recall from \cite{KocSch03} that 
 a countable family $\U$ of subsets of $X $ is  \emph{$\w$-groupable},
 if there exists a sequence $\la\D_n:n\in\w\ra$ of mutually disjoint finite subsets of $\U$ such that the set  $\{n\in\w:x\not\in\cup\D_n\}$ is 
 finite for all $x\in X$. 
 We extend our list of notation for specific covers of a space $X$ as follows:
 \begin{itemize}
\item $\C_{\wgp}(X)$ is the family of all
  closed $\w$-groupable  covers of $X$;
\item $\C^o_{\wgp}(X)$ is the family of all
  clopen $\w$-groupable  covers of $X$;
  \item $\B_{\wgp}(X)$ is the family of all
  Borel $\w$-groupable  covers of $X$.
\end{itemize}

%
%%%%%%%%%%%%%%%%%%%%%%%%%%%%%%%%%%%%%%%%%%%%%%%%%%%%%%%%%%%%%%%%%%%%%%%%%%

\begin{corollary}  \label{cor:ness_c_p_Hnw}
Let $X$ be a metrizable separable zero-dimensional space and $\CS$ a  countable clopen base of
$X$ closed under finite unions and complements. 
If $Y=Y_{\CS}$ is H-{\it nw}-selective\  as a subspace of $C_p(X,2)$, then 
all finite powers of
$X$  satisfy $U_{\mathit{fin}}(\Op,\Gamma)$ (i.e., are Hurewicz)
and 
$X$ satisfies $U_{\mathit{fin}}(\B,\B_\Gamma)$ (i.e., is Hurewicz with respect to all countable Borel covers), which is equivalent to $S_1(\B_\Gamma,\B_\Gamma)$.
\end{corollary}
\begin{proof}
 Let $\la\C_n:n\in\w\ra$  be a sequence of countable closed $\w$-covers of $X$.
 Theorem~\ref{thm:ness_c_p_Hnw} yields a 
 a $\gamma_{\mathit{fs}}$-sequence $\la\D_n:n\in\w\ra$ on $X$ such that 
$\D_n\in[\C_n]^{<\w}$. It follows that $\{\cup\D_n:n\in\w\}\in \C_\Gamma(X)$.
Thus, $X$ satisfies $U_{\mathit{fin}}(\C_\Omega,\C_\Gamma)$, which is 
obviously equivalent to $U_{\mathit{fin}}(\C,\C_\Gamma)$, i.e., the
 Hurewicz covering property with respect to countable closed 
covers.  \cite[Theorem~5.2]{BRR} implies that 
$X$ satisfies the Hurewicz property with respect to countable Borel covers,
which is  equivalent to $S_1(\B_\Gamma,\B_\Gamma)$
by \cite[Theorem~1]{SchTsa02}.

Finally, to see that all finite powers of $X$ are Hurewicz, note that Theorem~\ref{thm:ness_c_p_Hnw} implies that for every 
sequence $\la\C_n:n\in\w\ra$ of countable closed $\w$-covers of $X$
there exists a $\gamma_{\mathit{fs}}$-sequence $\la\D_n:n\in\w\ra$ on $X$ such that 
$\D_n\in[\C_n]^{<\w}$ and $\D_{n_0}\cap \D_{n_1}=\emptyset$
for any natural numbers $n_0\neq n_1$. Indeed, for this it is enough 
replace $\la\C_n:n\in\w\ra$ with a sequence $\la\C'_m:m\in\w\ra$ of countable closed $\w$-covers of $X$ such that for every cofinite subset $C$ of some $\C_n$,  there exists $m\in\w$ with $\C=\C'_m$. 
Theorem~\ref{thm:ness_c_p_Hnw} implies that there exists a $\gamma_{\mathit{fs}}$-sequence $\la\D'_m:m\in\w\ra$ on $X$ such that 
$\D'_m\in[\C'_m]^{<\w}$.
Now it is easy to see that one can choose a 
subsequece $\la\D_n:n\in\w\ra$ of $\la\D'_m:m\in\w\ra$
consisting of mutually disjoint elements and
such that 
$\D_n\in[\C_n]^{<\w}$. 

Using the  consequence of Theorem~\ref{thm:ness_c_p_Hnw} established in 
the paragraph above for sequences 
of countable clopen covers of $X$, and the obvious fact that if
 a $\gamma_{\mathit{fs}}$-sequence $\la\D_n:n\in\w\ra$ consists of 
mutually disjoint finite sets, then $\bigcup_{n\in\w}\D_n$ is an $\w$-groupable cover of $X$, we conclude that $X$ satisfies $S_{\mathit{fin}}(\C_\Omega,
\C_{\wgp})$, and hence also $S_{\mathit{fin}}(\C^o_\Omega,
\C^o_{\wgp})$.
By  
\cite[Theorem~16]{KocSch03}
  all finite powers of $X$ are Hurewicz.
\end{proof}

%%%%%%%%%%%%%%%%%%%%%%%%%%%%%%%%%%%%%%%%%%%%%%%%%%%%%%%%%%%%%%%%%%

By arguments 
similarly to (but easier than) those used in the proofs of
Theorem~\ref{thm:ness_c_p_Hnw} and Corollay~\ref{cor:ness_c_p_Hnw}, we can also get  necessary conditions for countable dense subsets of $C_p(X)$
of the form $Y_\B$
to be
M-{\it nw}-selective\  and R-{\it nw}-selective.

\begin{theorem} \label{thm:ness_c_p_Mnw}
Let $X$ be a metrizable separable zero-dimensional space and $\CS$ a  countable clopen base of
$X$ closed under finite unions and complements. 
If $Y=Y_{\CS}$ is M-{\it nw}-selective\  as a subspace of $C_p(X,2)$, then  $X$
satisfies $S_{\mathit{fin}}(\C_\Omega,\C_\Omega)$.
\end{theorem}

\begin{corollary}  \label{cor:ness_c_p_Mnw}
Let $X,\CS$ be such as in Theorem~\ref{thm:ness_c_p_Mnw}. 
If $Y=Y_{\CS}$ is M-{\it nw}-selective\  as a subspace of $C_p(X,2)$, then all finite powers of
$X$ are Menger, and also $X$ has the Menger property with respect to countable closed covers.
\end{corollary}
\begin{proof}
Clearly, $S_{\mathit{fin}}(\C_\Omega,\C_\Omega)$
implies $U_{\mathit{fin}}(\C,\C)$, i.e.,
the Menger property with respect to all countable closed covers

Also, $S_{\mathit{fin}}(\C_\Omega,\C_\Omega)$
implies $S_{\mathit{fin}}(\C^o_\Omega,\C^o_\Omega)$,
which for zero-dimensional spaces is equivalent to all finite powers
having the Menger property $U_{\mathit{fin}}(\Op,\Op)$
by \cite[Theorem~3.9]{JMS}.
\end{proof}

%%%%%%%%%%%%%%%%%%%%%%%%%%%%%%%%%%%%%%%%%%%%%%%%%%%%%%%%%%%%%%%%%%%%%

\begin{theorem} \label{thm:ness_c_p_Rnw}
Let $X,\CS$ be such as in Theorem~\ref{thm:ness_c_p_Mnw}.
If $Y=Y_{\CS}$ is R-{\it nw}-selective\  as a subspace of $C_p(X,2)$, then  $X$
satisfies $S_1(\C_\Omega,\C_\Omega)$.
\end{theorem}

\begin{corollary}  \label{cor:ness_c_p_Mnw}
Let $X,\CS$ be such as in Theorem~\ref{thm:ness_c_p_Mnw}. 
If $Y=Y_{\CS}$ is R-{\it nw}-selective\  as a subspace of $C_p(X,2)$, then all finite powers of
$X$ satisfy $S_1(\Op,\Op)$, i.e., are Rothberger, and also $X$ has the Rothberger property $S_1(\C,\C)$ with respect to countable closed covers.
\end{corollary}
\begin{proof}
By Theorem~\ref{thm:ness_c_p_Rnw} we know that 
$X$ satisfies $S_1(\C_\Omega,\C_\Omega)$, and hence it also 
satisfies satisfies $S_1(\C_\Omega,\C)$, which is equivalent
to  $S_1(\C,\C)$ by the same argument as in the proof of
\cite[Theorem~17]{Sch1}, which asserts that $S_1(\Omega,\Op)$
and $S_1(\Op,\Op)$ are equivalent.

Also, by Theorem~\ref{thm:ness_c_p_Rnw} the space 
$X$ satisfies $S_1(\C_\Omega,\C_\Omega)$, and
hence also $S_1(\C^o_\Omega,\C^o_\Omega)$, which is obviously equivalent 
to $S_1(\Omega,\Omega)$ because $X$ is zero-dimensional.
Finally, $S_1(\Omega,\Omega)$  is equivalent to all 
finite powers being Rothberger, see \cite[Lemma, p.~918]{S0} or \cite[Theorem~3.8]{JMS}.
\end{proof}

%%%%%%%%%%%%%%%%%%%%%%%%%%%%%%%%%%%%%%%%%%%%%%%%%%%%%%%%%%%%%%%%%%%%%%%%
The necessary conditions proved above motivate the following question.

\begin{question} \label{que:cl_vs_bor}
Let $X$ be a metrizable separable zero-dimensional space.
\begin{enumerate}
\item Suppose that $X$ is Menger with respect to countable closed covers.
Is $X$ Menger with respect to countable Borel covers?
\item Suppose that $X$ is Rothberger with respect to countable closed covers.
Is $X$ Rothberger with respect to countable Borel covers?
\end{enumerate}
\end{question}

As we have already mentioned in the proof of Corollary~\ref{cor:ness_c_p_Hnw},
by \cite[Theorem~5.2]{BRR} the answer to the analogous question for 
the Hurewicz property is affirmative.
Regarding the Rothberger part of Question~\ref{que:cl_vs_bor},
in the Laver model all Rothberger metrizable spaces are countable, hence 
the affirmative answer is consistent, which means that this question is interesting in models where the Borel conjecture fails, e.g., models of CH.
Below we show that also for the Menger part of Question~\ref{que:cl_vs_bor}
the affirmative answer is consistent.

\begin{proposition} \label{prp:mil_<d}
In the Miller model, if $X\subset 2^\w$ satisfies 
$U_{\mathit{fin}}(\C,\C)$, then $|X|<\hot d$, and hence $X$ satisfies the Menger property with respect to arbitrary countable covers.
\end{proposition}
\begin{proof}
 First we shall show that any $G\subset X$ is Menger. Indeed, let
   $\la\U_n:n\in\w\ra$ be a sequence of  covers of $G$
   by open subsets of $X$. For every
   $n$ let $\A_n$ be a countable cover of $G$ by open subsets of $X$
   such that for every $A\in\A_n$ there exists $U\in\U_n$
   with $\bar{A}\subset U$,  the closure being taken in $X$.
   Set 
   $$\W_n=\{\bar{A}:A\in\A_n\}\cup\{F_n\}, \mbox{ where } F_n=X\setminus\cup\A_n,$$
   and note that $\W_n$ is a countable closed cover of $X$.
   The Menger property for countable closed covers 
   applied to $X$ yields a sequence
   $\la\W'_n:n\in\w\ra$ such that $\W'_n\in [\W_n]^{<\w}$
   and $G\subset\bigcup_{n\in\w}\cup\W'_n$. Since each $F_n$ is disjoint from $G$, we get $G\subset\bigcup_{n\in\w}\cup\W''_n$,
   where $\W''_n=\W'_n\setminus\{F_n\}$. It follows that there exists 
   a finite $\A''_n\subset\A_n$ such that $\W''_n=\{\bar{A}:A\in\A''_n\}$.
Consequently, there exists a finite $\U''_n\subset\U_n$  such that 
$$\forall A\in\A''_n\: \exists \U\in\U''_n\: (\bar{A}\subset U).$$
Putting all together, we get $G\subset\bigcup_{n\in\w}\cup\U''_n$,
and therefore $G$ is Menger\footnote{Let us note that this part did not require any assumptions beyond ZFC.}.

In the Miller model for every Menger space $Z\subset 2^\w$ and a $G_\delta$-subset $H\subset 2^\w$, if $Z\subset H$, then there is a family
$\K$ of compact subspaces of $H$ of size $|\K|\leq\w_1$ and such that
$Z\subset\cup\K$, see \cite[Theorem~4.4]{MilTsaZdo14}. As a result, if $Q\in [2^\w]^\w$
is disjoint from $Z$, then there exists a $G_{\w_1}$-subset (i.e., an intersection of $\w_1$-many open sets)  
$R$ of $2^\w$ such that $Q\subset R$ and $R\cap Z=\emptyset$.

Since $X$ is hereditarily Menger, we conclude that for every
$Q\in [X]^\w$ there exists a $G_{\w_1}$-subset $R(Q)$ with $R(Q)\cap X=Q$.
Now a direct application of \cite[Lemma~2.5]{Zdo18} gives that there exists a 
family $\mathsf Q\subset[2^\w]^\w$ of size $|\mathsf Q|=\w_1$ and such that
$$X=\bigcup_{Q\in\mathsf Q}(R(Q)\cap X)=\bigcup_{Q\in\mathsf Q}Q=\cup\mathsf Q,$$
which yields $|X|\leq \w_1<\hot d$.
Finally, the fact that any space of size $<\hot d$ has the Menger property with respect to the family of  all countable covers 
is straightforward.
\end{proof}

The next statement is a consequence of \cite[Corollary~4.4]{MilTsa10}.

\begin{proposition} \label{prp:lav_<b}
In the Laver model, if $X\subset 2^\w$ satisfies $S_1(\B_\Gamma,\B_\Gamma)$, i.e., is Hurewicz with respect to the family of 
countable Borel covers, then $|X|<\hot b$.
\end{proposition}

As a conclusion we have the following fact showing that countable spaces considered in Theorems~\ref{thm:ness_c_p_Hnw}, \ref{thm:ness_c_p_Mnw} and 
\ref{thm:ness_c_p_Rnw} \emph{cannot} give  non-trivial examples satisfying corresponding
$\mathit{nw}$-selectivity properties in ZFC.

\begin{corollary} \label{cor:no-non_trv_zfc}
In the Miller (resp. Laver) model,
let $X$ be a metrizable separable zero-dimensional space and $\CS$ a  countable clopen base of
$X$ closed under finite unions and complements. 
If $Y_{\CS}$ is M-{\it nw}-selective\ (resp. H-{\it nw}-selective\  or R-{\it nw}-selective), then it is trivial,
i.e., $w(Y)=|X|\leq\w_1<\hot d$ (resp. $w(Y)=|X|\leq\w_1<\hot b$ or 
$w(Y)=|X|\leq\w<\mathit{cov}(\M)=\w_1$).
\end{corollary}

Even though spaces of the form $Y_{\CS}$ seem to be  
one's first inclination to construct countable dense subspaces of
$C_p(X,2)$, there are also other countable dense subspaces of 
$C_p(X,2)$, and we do not have any efficient way of analyzing 
their $nw$-selectivity properties in terms of (covering)
properties of $X$.

\begin{question}
Are there ZFC examples of  metrizable separable zero-dimensional spaces $X$ of size $\geq \hot d$ (resp. $\geq \hot b$, $\geq \mathit{cov}(\M)$) and countable dense subspaces $Y$ of
$C_p(X,2)$ which are M-{\it nw}-selective\ (resp. H-{\it nw}-selective, R-{\it nw}-selective)?
\end{question}

%%%%%%%%%%%%%%%%%%%%%%%%%%%%%%%%%%%%%%%%%%%%%%%%%%%%%%%%%%%
%%%%%%%%%%%%%%%%%%%%%%

The following fact has been established at the beginning of the proof of Proposition~\ref{prp:mil_<d} for Menger spaces without using any 
additional assumptions beyond ZFC,
and the same argument also works in  two other cases.

\begin{corollary} \label{cor:hered}
Let $X$ be a metrizable separable space.  If  
$X$ satisfies the Menger (resp. Hurewicz, Rothberger)
property for countable closed covers, then it is hereditarily
Menger (resp. Hurewicz, Rothberger).    
\end{corollary}

%%%%%%%%%%%%%%%%%%%%%%%%%%%%%%%%%%%%%%%%%%%%%%%%%%%%%%%%%%%%%
%%%%%%%%%%%%%%%%%%%%%%%%%%%%%%%%%%%%%%%%%%%%%%%%%%%%%%%%%%%%%%

In \cite{BG} it is provided an example distinguishing countable fan tightness and $M$-$nw$-selectivity which is uncountable,
now we can provide a countable one.

\begin{proposition} \label{Mnwsel}
The space $X=C_p(2^\omega,2)$ 
is countable $H$-separable and weakly Fr\'echet in the strict sense,
but it is not  M-{\it nw}-selective.
\end{proposition}  
\begin{proof}
	 $2^\omega$ is not hereditarily Menger since 
  $[\w]^\w\subset 2^\w$ is not Menger being a copy of the Baire space. 
 Thus, $C_p(2^\w,2)$ is not M-{\it nw}-selective\ by Corollary~\ref{cor:hered}.

 The other properties of $C_p(2^\w,2)$
 directly follow from   \cite[Theorem~40]{BBM}.
  \end{proof}

%%%%%%%%%%%%%%%%%%%%%%%%%%%%%%%%%%%%%%%%%%%%%%%%%%%%%%%%%%%%%%%%%%%%%%%
%%%%%%%%%%%%%%%%%%%%%%%%%%%%%%%%%%%%%%%%%%%%%%%%%%%%%%%%%%%%%%%%%%%%%%%

\section{Non-preservation by products}\label{sec:prod}

This section is devoted to the proof of the following

\begin{theorem} \label{thr:non-prod_hnw}
It is consistent that there exist two countable H-{\it nw}-selective\ spaces with 
non-M-{\it nw}-selective\ product.
\end{theorem}

We need the following variation of 
 Proposition~\ref{prop: Bre_random}. Let us note that $2^\w=\{0,1\}^\w$
 with the operation $+$ of the coordinate-wise 
 addition modulo $2$ is a compact Boolean topological group.

\begin{proposition}    \label{prop: Bre_random_var}
	Suppose that $V$ is a ground model of CH and 
 $\{z_\alpha:\alpha<\w_1\}$  is an enumeration in $V$ of
  $[\w]^\w\subset 2^\w$.
   Let $X=\{r_\alpha: \alpha<\w_1\}\subset 2^\w$ be the  set of  generic random reals over $V$  added by  $B(\w_1)$. Let also 
$G$ be $B(\w_1)$-generic over $V$ giving rise to $X$.
 Then in $V[G]$,  all finite power of 
 $$X_1=\{r_\alpha+z_\alpha:\alpha<\w_1\}\subset 2^\w$$
 satisfy $S_1(\B_\Gamma,\B_\Gamma)$.
 \end{proposition}

 Proposition~\ref{prop: Bre_random_var} is a 
consequence of the following fact which could be established in the same way
as Lemma~\ref{lem:coh_in_nonmeag}, basically replacing ``Cohen'' and ``meager''
with ``random'' and ``measure $0$''.

\begin{lemma} \label{lem:ran_in_nonzero_var}
We use notation from Proposition~\ref{prop: Bre_random_var}.
Suppose that $D\subset 2^\omega$ is a Borel non-measure zero set coded in $V$. Then there exists $\beta<\lambda $ such that $r_\beta+z_\alpha\in D$.
\end{lemma}
\smallskip

\noindent\textit{Proof of Theorem~\ref{thr:non-prod_hnw}.}\ 
We use   notation from Proposition~\ref{prop: Bre_random_var} and work in $V[G]$.
By the definitions of $X$ and $X_1$ we have 
that 
$$z_\alpha=(r_\alpha+z_\alpha)+r_\alpha\in X_1+X$$
for all $\alpha\in\w_1$,
and hence $[\w]^\w\cap V\subset X+X_1$. 
On the other hand, since for $\alpha_0\neq\alpha_1$ the sum $r_{\alpha_0}+r_{\alpha_1}$
cannot lie in $V$,
we conclude that $X+X_1\subset [\w]^\w$.
Thus 
$$ [\w]^\w\cap V\subset X+X_1\subset [\w]^\w. $$

Since $B(\w_1)$ does not add unbounded reals, $[\w]^\w\cap V$
is dominating, where each infinite subset $a$ of $\w$ is identified with the increasing function 
in $\w^\w$  whose range is $a$. 
It follows that $X\times X_1$ is not Menger since it has a dominating continuous image
in $[\w]^\w$, namely
$X+X_1$. Consequently, $[X\sqcup X_1]^2$ is not Menger because it has a closed topological copy of $X\times X_1$.

Now, by
Corollary~\ref{cor:ness_c_p_Mnw}, if $\CS$ is a countable clopen base 
for $X\sqcup X_1$ closed under finite unions and complements, than 
$Y_{\CS}$ is not M-{\it nw}-selective\  as a subspace of
$C_p(X\sqcup X_1,2)= C_p(X,2)\times C_p(X_1,2)$.

Let $\CS(X)$ and $\CS(X_1)$ be
countable clopen bases closed under finite unions and complements  
for $X$ and $X_1$, respectively. 
Then $Y_{\CS(X)}$ and $Y_{\CS(X_1)}$ are countable dense H-{\it nw}-selective\ 
subspaces of $C_p(X,2)$ and $C_p(X_1,2)$ by Theorem~\ref{thm:suff_hnw}, respectively.

On the other hand, set 
$$\CS=\{U\cup W: U\in\CS(X),W\in\CS(X_1)\}$$
and observe that $\CS$  is a countable clopen base 
for $X\sqcup X_1$ closed under finite unions and complements,
and hence $Y_{\CS}$ is not M-{\it nw}-selective\ as a subspace of $C_p(X\sqcup X_1,2)$.
It remains to 
note that 
$Y_{\CS}$ is a homeomorphic copy of
$Y_{\CS(X)}\times Y_{\CS(X_1)}$.
\hfill $\Box$
\medskip

The analogous strategy with random reals replaced 
by Cohen reals does not seem to give anything interesting: Unlike 
 in the random model, $[\w]^\w\cap V$ is known to satisfy
 $S_1(\B_\Omega,\B_\Omega)$ in the Cohen model, so the approach above 
 based on $[\w]^\w\cap V$ being ``big'' in a suitable sense (e.g., dominating in the random model)
 does not work.  This motivates the following 

 \begin{question} \label{q5}
 Is the existence of two countable
 R-{\it nw}-selective\ spaces with non-R-{\it nw}-selective\ (or even non-M-{\it nw}-selective) product consistent?
  \end{question}

On the other hand, we do not know whether countable spaces like in Theorem~\ref{thr:non-prod_hnw}
could be constructed in ZFC.

\begin{question}\label{q6}
Is it consistent that the product of two countable
M-{\it nw}-selective\ (resp. H-{\it nw}-selective, R-{\it nw}-selective) spaces is again M-{\it nw}-selective\ (resp. H-{\it nw}-selective, R-{\it nw}-selective)?
\end{question}

%%%%%%%%%%%%%%%%%%%%%%%%%%%%%%%%%%%%%%%%%%%%%%%%%%%%%%%%%%%%%%%%%%%%%%%

\section{HFD's and R-{\it nw}-selectivity}\label{sec:HFD}

HFD spaces where introduced in order to construct $S$-spaces, i.e.,  hereditarily separable spaces which are not Lindel\"of, see \cite{J, Juh} and references therein.
In this section we show that stronger version of the HFD spaces are 
\Rnw. The following definition is taken from \cite{J}.

\begin{definition} \label{def_7_1}
 Let $\lambda$ be an uncountable cardinal.
 A subset $X \subset 2^\lambda$ is called HFD (abbreviated from Hereditarily Finally Dence) if $X$ is infinite and for every $A \in[X]^\omega$ there is a $B \in[\lambda]^\omega$ such that $A$ (i.e., $A \upharpoonright(\lambda \backslash B)$) is dense in $2^{\lambda \backslash B}$.
 %(such an $A$ is also called finally dense).
\end{definition}
We use the following notation of \cite{Juh}. Given some $\varepsilon \in 
\mathit{Fin}(I, 2)$, where $\mathit{Fin}(I,2)$ denotes the collection of all finite partial functions on $I$ to $2$, we set $[\varepsilon]=\{f \in 2^I: \varepsilon \subset f\}$. Thus, $[\varepsilon]$ is a standard  basic clopen subset of $2^I$. Now suppose that  $I$ is a set of ordinals,  $b \in[I]^{<\omega}$, $ b=\left\{\beta_i: i \in n=|b|\right\}$ is an  increasing enumeration, and $\varepsilon \in 2^n$.  In this case we denote by $\varepsilon * b$ the element of $\mathit{Fin}(I,2)$ which has $b$ as its domain and satisfies $\varepsilon * b\left(\beta_i\right)=\varepsilon(i)$ for all $i \in n$. 

For any infinite cardinal $\mu$ and $r \in \omega$ we denote by $\mathcal{D}_\mu^r(I)$ the collection of all sets $B \in\left[[I]^r\right]^\mu$ such that the members of $B$ are pairwise disjoint. We shall write
$$
\mathcal{D}_\mu(I)=\bigcup\left\{\mathcal{D}_\mu^r(I): r \in \omega\right\}
$$
If $B \in \mathcal{D}_\mu(I)$ then $n(B)=|b|$ for any $b \in B$. Now, if $B \in \mathcal{D}_\mu(I)$ and $\varepsilon \in 2^{n(B)}$ then
$$
[\varepsilon, B]=\bigcup\{[\varepsilon * b]: b \in B\}
$$
is called a $\mathcal{D}_\mu$-set in $2^I$. The most important instance of the above is  $\mu=\omega$, in this case we shall often omit the lower index $\omega$ of $\mathcal{D}$, i.e.,  a $\mathcal{D}$-set in $2^I$ is a $\mathcal{D}_\omega$-set. Clearly, any $\mathcal{D}$-set is open dense  and has product (Lebesgue) measure 1 in $2^I$. 
Recall that a map $F: \kappa \times \lambda \rightarrow 2$ with $\kappa \geqslant \omega, \lambda \geqslant \omega_1$ is called an HFD matrix (see \cite{Juh}) if for every $A \in[\kappa]^\omega, B \in \mathcal{D}_{\omega_1}(\lambda)$ and $\varepsilon \in 2^{n(B)}$ there are $\alpha \in A$ and $b \in B$ such that
$$
f_\alpha=F(\alpha,-) \supset \varepsilon * b .
$$
The latter inclusion means that $F\left(\alpha, \beta_i\right)=\varepsilon(i)$  for each $i<n(B)=$ $|b|$, where $\beta_i$ is the $i$-th member of $b$ in its increasing enumeration. The following fact was established in \cite{Juh}.

\begin{proposition}
	$X \subset 2^\lambda$ is an HFD space if and only if there exists an HFD matrix $F: \kappa \times \lambda \rightarrow 2$ such that $X=\left\{f_\alpha: \alpha \in \kappa\right\}$. 
\end{proposition}
In this case we say that $F$ \emph{represents} $X$. 
Following \cite{Juh}, for an HFD space $X \subset 2^\lambda$ and $A \in[X]^\omega$ we set
$${\J}(A)=\big\{I \in[\lambda]^\w: \forall \varepsilon \in \mathit{Fin}(I,2)
\: \big(|A \cap [\varepsilon]|=\omega \:\Rightarrow\: A \cap[\varepsilon] \text { is dense in } 2^{\lambda \backslash I}\big)\big\}.$$

\begin{proposition}\cite{Juh}
	If $X \subset 2^\lambda$ is HFD and $A \in[X]^\omega$ then $\mathcal{J}(A)$ is closed and unbounded in $[\lambda]^\omega$.
\end{proposition}

\begin{proposition} \label{net_impl_fin}
	If $\NN$ is a countable network in a HFD space $X\subset 2^\lambda$, then
	$\NN\cap [X]^{<\w}$ is a network in $X$ as well, and hence $X$ is countable.
\end{proposition} 
\begin{proof}
	Set $\A=\NN\cap [X]^{\geq\w}$, i.e., $\A$ is a family of all infinite elements of
	$\NN$. For every $A\in\A$ fix a countable infinite $C(A)\subset A$ and pick
	$J\in\bigcap_{A\in\A}\J(C(A))$.
Fix  any $\alpha\in J$, $\beta\in\lambda\setminus J$,
	$x\in X$ and set  $\epsilon_0=\{\la\alpha,x(\alpha)\ra\} \in  \mathit{Fin}(J,2)$,
	$\epsilon_1=\{\la\alpha,x(\alpha)\ra, \la\beta , 1- x(\beta)\ra \}\in
 \mathit{Fin}(\lambda,2)$, and
	$\epsilon_2=\{\la\alpha,x(\alpha)\ra, \la\beta , x(\beta)\ra \}\in\mathit{Fin}(\lambda,2)$.
	
	We claim that $C(A)\not\subset [\epsilon_2]$ for any $A\in\A$, which would clearly imply that
	$\NN\cap [X]^{<\w}$ must be a network for $X$. Fix $A\in\A$. If $C(A)\not\subset [\epsilon_0]$ there is nothing to prove,
	so assume that $C(A)\subset [\epsilon_0]$. But then $C(A)\cap [\epsilon_1]\neq\emptyset$
	because $J\in\J(C(A))$, and hence
	$C(A)\not\subset [\epsilon_2]$.
\end{proof}

Scheepers \cite{Sch2} proved that any HFD is R-separable. Now we show that the following stronger version of HFD spaces introduced in \cite{SSS} implies R-{\it nw}-selectivity.

\begin{definition}\label{verystrongHFD}
	A set $X \subset 2^{\omega_1}$ is a \emph{very strong HFD} if for each sequence $\left\{A_n: n \in \omega\right\}$ of pairwise disjoint, non-empty finite subsets of $X$ there is $\beta<\omega_1$ such that for all $s \in \mathit{Fin}(\omega_1 \backslash \beta, 2)$ there are infinitely many $n$ with $A_n \subset[s]$.
\end{definition}

Obviously, we get an equivalent notion if we require in the definition above only that
$\{A_n:n\in\omega\}\cap [[s]]^{<\omega} \not=\emptyset$.

Recall that a dense set $D \subseteq X$ is called \emph{groupable} if it admits a partition ${\A} =\{ A_n: n<\omega\}$ into finite sets such that every non-empty open subset of $X$ meets all but finitely elements of ${\A}$. Every very strong HFD space cannot contain a groupable dense subset \cite{SSS}. On the other hand, every H-separable space has a groupable dense subset. Therefore a very strong HFD space cannot be H-separable (hence not H-{\it nw}-selective).

\begin{theorem} \label{strongHFD:th}
	Every countable very strong HFD  space $X$ is R-{\it nw}-selective.
\end{theorem}
\begin{proof} 
	Let  $({\NN}_{n, m,k}: n, m,k \in \omega)$ be an enumeration of countably many  networks of  $X$. By Proposition~\ref{net_impl_fin} there is no loss of generality in assuming that 
 each $\NN_{n,m,k}$ consists of finite subsets of $X$.

  Let $\vartheta$ be a large enough regular cardinal and  $M$  a countable elementary submodel of $\mathcal{H}_{\vartheta}$ which contains $X$ and $(\NN_{n, m,k}: n, m,k \in \omega)$. Set $\beta=M \cap \omega_1$ and enumerate $\mathit{Fin}(\beta, 2)$ as $\left\{s_n: n \in \omega\right\}$. 
  Let $L=\{n\in\w:[s_n]\cap X\neq\emptyset\}$ and
  for all $n \in L$ fix a (non-necessary injective) enumeration  $\{x^n_k:k\in\omega\}$ of $[s_n]\cap X$.
  Given $n\in L$ and $k\in\w$, 
  by induction on $m\in\w$  choose $N_{n, m,k} \in [[s_n]]^{<\omega} \cap \NN_{n, m,k}$ such that 
  \begin{itemize}
 \item   $x^n_k\in N_{n,m,k}$,
 \item $(N_{n,m_0,k}\setminus\{x^n_k\})\cap (N_{n,m_1,k}\setminus\{x^n_k\})=\emptyset $
 for any $m_0<m_1$,   and
 \item  the function $m \mapsto N_{n, m, k}$ is in $M$.
	\end{itemize}
 We claim that $\{N_{n, m,k}: n\in L, m,k\in\omega\}$ is a network in $X$. Indeed, let $s \in Fin(\omega_1, 2)$, $x\in [s]$, and note that $s \restriction_\beta=s_n$ for some $n$. Thus $n\in L$. 
 Fix $k\in\omega$ such that $x=x^n_k$. The sequence 
 $$(N_{n, m,k}\setminus\{x^n_k\}: m \in \omega)$$
 lies in $M$ and consists of mutually disjoint finite subsets of $X$,
 so there is $\beta_n \in M$ such that 
 $$[[t]]^{<\omega} \cap \{N_{n, m,k}\setminus\{x^n_k\}: m \in \omega\} $$
 is infinite
  for all $t \in \mathit{Fin}(\omega_1 \backslash \beta_n, 2)$ with $x=x^n_k\in [t]$. This is a direct consequence of  $X$ being a  very strong HFD.
It follows  that 
  $$[[t]]^{<\omega} \cap \{N_{n, m,k}: m \in \omega\} $$
 is infinite
  for all $t \in \mathit{Fin}(\omega_1 \backslash \beta_n, 2)$ with $x=x^n_k\in [t]$. 
    Note that $s \backslash s_n \in \mathit{Fin}(\omega_1 \backslash \beta, 2) \subset \mathit{Fin}(\omega_1 \backslash \beta_n, 2)$ hence $[[s \backslash s_n]]^{<\omega} \cap \{N_{n, m,k}: m \in \omega\} \neq \emptyset$. Since $ \{N_{n, m,k}: m \in \omega\} \subset[[s_n]]^{<\omega}$ we have that $[[s]]^{<\omega} \cap \{N_{n, m,k}: m \in \omega\} \neq \emptyset$ as well.	
This completes our proof.
\end{proof}

Theorem~\ref{strongHFD:th} motivates the following
\begin{question}
 Is every countable HFD space \Rnw?   
\end{question}

\bigskip
{\bf Acknowledgement:} The first two authors in alphabetical order would like to thank the \lq\lq National Group for Algebric and Geometric Structures, and their Applications\rq\rq (GNSAGA-INdAM) for their invaluable support throughout the course of this research.
The research of the fourth author was funded in whole by the Austrian Science Fund (FWF) [I 5930].

%%%%%%%%%%%%%%%%%%%%%%

\end{document}